\documentclass[a4paper,reqno,11pt, oneside]{amsart}

\usepackage{amssymb}
\usepackage{amstext}
\usepackage{amsmath}
\usepackage{amscd}
\usepackage{amsthm}
\usepackage{amsfonts}
\usepackage{enumerate}
\usepackage{graphicx}
\usepackage{color}
\usepackage{here} % 図や表を強制的に出力させる
\usepackage{bm} % 太字ベクトル
\usepackage{tikz}

\makeatletter
  
  \@addtoreset{equation}{section}
\makeatother

\newcommand{\calB}{\mathcal{B}}

\newcommand{\calF}{\mathcal{F}}
\newcommand{\calI}{\mathcal{I}}
\newcommand{\calK}{\mathcal{K}}

\newcommand{\calP}{\mathcal{P}}

\newcommand{\calS}{\mathcal{S}}

\newcommand{\ZZ}{\mathbb{Z}}

\newcommand{\RR}{\mathbb{R}}
\newcommand{\CC}{\mathbb{C}}

\newcommand{\eb}{\mathbf{e}}

\newcommand{\vb}{\mathbf{v}}

\newcommand{\xb}{\mathbf{x}}
\newcommand{\yb}{\mathbf{y}}

\def\opn#1#2{\def#1{\operatorname{#2}}} % to make operators
\opn\conv{conv} \opn\mut{mut} \opn\GL{GL} \opn\cone{cone} \opn\ini{in} \opn\NF{NF} \opn\deg{deg}
\opn\NF{NF} \opn\sign{sign} \opn\mat{Mat} \opn\rank{rank} \opn\PC{PC}
\newcommand{\FB}{\calF\calB}

\newtheorem{thm}{Theorem}[section]
\newtheorem{cor}[thm]{Corollary}
\newtheorem{lem}[thm]{Lemma}
\newtheorem{prop}[thm]{Proposition}
\newtheorem{problem}[thm]{Problem}

\theoremstyle{definition}
\newtheorem{defi}[thm]{Definition}
\newtheorem{ex}[thm]{Example}

\theoremstyle{remark}
\newtheorem{rem}[thm]{Remark}

\tikzset{
	hollow node/.style={circle,draw,inner sep=2.0}
}                                                                                   
\tikzset{                                                                           
	level 1/.style = {hollow node, level distance=1.2cm, sibling distance=1.6cm}    
}                                                                                   
\tikzset{                                                                           
	level 2/.style = {hollow node, level distance=1.2cm, sibling distance=1.2cm}    
}                                                                                   
\tikzset{                                                                           
	level 3/.style = {hollow node, level distance=1.2cm,sibling distance=0.8cm}     
}                                                                                   
\tikzset{                                                                           
	level 4/.style = {hollow node, level distance=1.2cm,sibling distance=0.4cm}     
}                                                                                   
\tikzset{                                                                           
	end/.style = {hollow node}                                                      
}

\begin{document}

\title{Cohomological rigidity for Fano Bott manifolds}
\author{Akihiro Higashitani}
\author{Kazuki Kurimoto}

\address[A. Higashitani]{Department of Pure and Applied Mathematics, Graduate School of Information Science and Technology, Osaka University, Suita, Osaka 565-0871, Japan}
\email{higashitani@ist.osaka-u.ac.jp}
\address[K. Kurimoto]{Department of Mathematics, Graduate School of Science, Kyoto Sangyo University, Kyoto 603-8555, Japan}
\email{i1885045@cc.kyoto-su.ac.jp}

\subjclass[2010]{
Primary 57R19; %Algebraic topology on manifolds
Secondary 14M25, %Toric varieties, Newton polyhedra
14J45, %Fano varieties
57S15. %Compact Lie groups of differentiable transformations
} 
\keywords{Cohomological rigidity, toric Fano manifold, Bott manifold,}

\maketitle

%%%---abstract---%%%
\begin{abstract} 
In the present paper, we characterize Fano Bott manifolds up to diffeomorphism in terms of three operations on matrix. 
More precisely, we prove that given two Fano Bott manifolds $X$ and $X'$, the following conditions are equivalent: 
\begin{itemize}
\item[(1)] the upper triangular matrix associated to $X$ can be transformed into that of $X'$ by those three operations; 
\item[(2)] $X$ and $X'$ are diffeomorphic; 
\item[(3)] the integral cohomology rings of $X$ and $X'$ are isomorphic as graded rings. 
\end{itemize}
As a consequence, we affirmatively answer the cohomological rigidity problem for Fano Bott manifolds. 
\end{abstract}

\section{Introduction}

\subsection{Cohomological rigidity problem}

A toric variety is a normal complex algebraic variety with an algebraic torus of a $\CC^*$-torus which has an open dense orbit. 
In what follows, we call a compact smooth toric variety a toric \textit{manifold}. 

A fundamental fact on toric varieties claims that each toric variety one-to-one corresponds to a combinatorial object, called a \textit{fan}. 
Moreover, the compactness and the smoothness of toric varieties can be interpreted in terms of the associated fans. 
It is well known that the set of toric manifolds up to algebraic varieties one-to-one corresponds to the set of complete nonsingular fans up to unimodular equivalence. 
Moreover, it is also known by Batyrev \cite{Bat99} that toric varieties are isomorphic as algebraic varieties if and only if 
there is a bijection between primitive collections which preserves their associated primitive relations (See Subsection~\ref{sec:prim} for the details of them). 
In particular, the classification of toric manifolds as algebraic varieties is completely done by using combinatorial objects. 

However, the classification of toric manifolds as smooth manifolds is not established yet. 
Inspired by a classification of a certain class of Bott manifolds up to diffeomorphism in \cite{MP08}, the following naive problem was proposed: 
\begin{problem}[{cf. \cite{MS08}, Cohomological rigidity problem for toric manifolds}]\label{toi}
Are two toric manifolds diffeomorphic (or homeomorphic) if their integral cohomology rings are isomorphic as graded rings? 
\end{problem}
We denote the integral cohomology ring $H^*(X; \ZZ)$ of a toric manifold $X$ by $H^*(X)$. In what follows, we omit ``integral'' of integral cohomology. 

We say that a family $\calF$ of toric manifolds is \textit{cohomologically rigid} if any two toric manifolds in $\calF$ 
whose cohomology rings are isomorphic as graded rings are diffeomorphic (or homeomorphic).

No counterexample of the cohomological rigidity problem for toric manifolds is known. 
Towards the solution of Problem~\ref{toi}, many results have been obtained, all of which affirmatively answer the problem. 
Many results are related to Bott manifolds (\cite{Choi15, CLMP, CMM, CMO17, CMS101, CMS102, MP08, MS08}). Those will be explained below. 

\subsection{Bott manifolds and toric Fano manifolds}

The main object of the present paper is Bott manifolds. 
Note that if $B$ is a toric manifold and $E$ is a Whitney sum of complex line bundles over $B$, 
then the projectivization $P(E)$ of $E$ is again a toric manifold. 
Starting with $B$ as a point and repeating this construction, say $d$ times, we obtain a sequence of toric manifolds as follows: 
\begin{align*}
B_d \xrightarrow{\pi_d} B_{d-1} \xrightarrow{\pi_{d-1}} \cdots \xrightarrow{\pi_2} B_1 \xrightarrow{\pi_1} B_0 = \{\text{a point}\}, 
\end{align*}
where the fiber of $\pi_i : B_i \rightarrow B_{i-1}$ for $i = 1, \ldots, d$ is a complex projective space ${\CC}P^{n_i}$. 
This sequence is called a \textit{generalized Bott tower} of height $d$, and 
we call $B_d$ a \textit{$d$-stage generalized Bott manifold}. 
We say that $B_d$ is \textit{$d$-stage Bott manifold} (omitted ``generalized'') when $n_i=1$ for every $i$.

Bott manifolds are very well-studied objects in the area of toric topology. 
In fact, there are many results on cohomological rigidity problem for (generalized) Bott manifolds. 
The following theorems are a part of the known results on cohomological rigidity problem for toric manifolds concerning with (generalized) Bott manifolds. 
\begin{thm}[{\cite[Theorem 1.1]{CMS102}}]
Let $X$ be a $d$-stage generalized Bott manifold. If $H^*(X)$ is isomorphic to $H^*(\prod_{i=1}^d{\CC}P^{n_i})$ as graded rings, 
then every fibration is topologically trivial; in particular, $X$ is diffeomorphic to $\prod_{i=1}^d{\CC}P^{n_i}$. 
\end{thm}
\begin{thm}[{\cite{Choi15, CMS102}}]\label{thm:Bott}
The following results are known: 
\begin{itemize}
\item $2$-stage generalized Bott manifolds are cohomologically rigid (\cite[Theorem 1.3]{CMS102}); 
\item $3$-stage Bott manifolds are cohomologically rigid (\cite[Theorem 1.3]{CMS102}); 
\item $4$-stage Bott manifolds are cohomologically rigid (\cite[Theorem 3.3]{Choi15}). 
\end{itemize}
\end{thm}
Note that Theorem~\ref{thm:Bott} is still open for $d$-stage Bott manifolds with $d \geq 5$. 

The following is one of the most important contributions to cohomological rigidity problem for Bott manifolds: 
\begin{thm}[{\cite[Theorem 1.1]{CMM}}]
Any graded ring isomorphism between the cohomology rings of two Bott manifolds preserves their Pontrjagin classes. 
\end{thm}

\smallskip

%To investigate toric manifolds, primitive collections and primitive relations play a crucial role. 
%In fact, it is known by Batyrev \cite{Bat99} that toric varieties are isomorphic as algebraic varieties if and only if 
%there is a bijection between primitive collections which preserves their associated primitive relations. 
%For example, the classifications of toric $d$-folds with Picard number $\leq 3$ (\cite{Bat91}) and toric Fano $4$-folds (\cite{Bat99, Sato}) were performed in terms of primitive relations. 
%Anyway, the classification of toric manifolds as algebraic varieties can be completely done by using combinatorial objects. 
Many researchers also investigate toric \textit{Fano} manifolds. 
Here, we say that a nonsingular projective variety is \textit{Fano} if its anticanonical divisor is ample. 
Note that there are only finitely many toric Fano $d$-folds for each fixed $d$. 
The classification problem of toric Fano $d$-folds is of particular interest in the study of toric Fano manifolds. 
The classification of toric Fano $4$-folds was accomplished by Batyrev \cite{Bat99} and Sato \cite{Sato}. 
The key tool for this classification is \textit{primitive collections} and \textit{primitive relations} (See Subsection~\ref{sec:prim}). 
After their classification, {\O}bro \cite{Oeb} succeeded in constructing an algorithm, called \textit{SFP algorithm}, 
which produces the complete list of smooth Fano $d$-polytopes for a given positive integer $d$. 
For small $d$'s, the database of all smooth Fano $d$-polytopes is available in the following web page: 
\begin{center}{\tt http://www.grdb.co.uk/forms/toricsmooth} \end{center}
The classification problem of toric Fano manifolds was solved in some sense. 

By using those classifications, the authors and Masuda proved the following: 
\begin{thm}[{\cite[Theorem 1.1]{HKM}}]
Toric Fano $d$-folds with $d \leq 4$ are cohomologically rigid except for two toric Fano $4$-folds. 
\end{thm}
Note that the exceptional toric Fano $4$-folds have ID 50 and ID 57 in the above web page.

Furthermore, the following has been also proved in \cite{CLMP} by Cho, Lee, Masuda and Park: 
\begin{thm}[{\cite[Thereom 1.2]{CLMP}}]
For two Fano Bott manifolds $X$ and $X'$, if there is an isomorphism between their cohomology rings 
which preserves the first Chern class, then $X$ and $X'$ are isomorphic as algebraic varieties. 
\end{thm}

\smallskip

In \cite{KM} and \cite{CMO17}, cohomological rigidity problem for real Bott manifolds is investigated. 
Here, we call a manifold $X$ a \textit{real Bott manifold} if there is a sequence of ${\RR}P^1$ bundles such that 
for each $j=1,\ldots,d$, $B_j \rightarrow B_{j-1}$ is the projective bundle of the Whitney sum of a real line bundle and the trivial real line bundle over $B_{j-1}$, where $B_0$ is a point. 
It is well known that any $d$-stage real Bott manifold is determined by a certain upper triangular matrix with its entry in $\ZZ/2=\{0,1\}$, called a \textit{Bott matrix}. 
Moreover, in \cite[Section 3]{CMO17}, three operations for Bott matrices are introduced. We say that two Bott matrices are \textit{Bott equivalent} if those can be transformed by those three matrix operations. 

The following theorem gives a complete characterization of real Bott manifolds in terms of Bott matrices and answers the cohomological rigidity problem for real Bott manifolds in some certain sense. 
\begin{thm}[{\cite[Theorem 1.1]{CMO17}}, Classification of real Bott manifolds]\label{thm:realBott}
Let $A$ and $B$ be Bott matrices and let $X(A)$ and $X(B)$ be the associated real Bott manifolds to $A$ and $B$, respectively. 
Then the following three conditions are equivalent: 
\begin{itemize}
\item[(1)] $A$ and $B$ are Bott equivalent; 
\item[(2)] $X(A)$ and $X(B)$ are affinely diffeomorphic; 
\item[(3)] $H^*(X(A)) \otimes_\ZZ \ZZ/2$ and $H^*(X(B)) \otimes_\ZZ \ZZ/2$ are isomorphic as graded rings. 
\end{itemize}
\end{thm}
Note that the equivalence of (2) and (3) was originally proved in \cite[Theorem 1.1]{KM}. 

Our motivation to organize the present paper is to obtain an analogue of Theorem~\ref{thm:realBott} for Fano Bott manifolds.

\subsection{Main Result}

The main result of the present paper is the following: 
\begin{thm}[Main Result]\label{thm:FanoBott}
Let $X$ and $X'$ be Fano Bott manifolds and let $A(X)$ and $A(X')$ be the upper triangular matrices associated to $X$ and $X'$, respectively. 
Then the following three conditions are equivalent: 
\begin{itemize}
\item[(1)] $A(X)$ and $A(X')$ are Fano Bott equivalent; 
\item[(2)] $X$ and $X'$ are diffeomorphic; 
\item[(3)] $H^*(X)$ and $H^*(X')$ are isomorphic as graded rings. 
\end{itemize}
\end{thm}
See Section~\ref{sec:matrix} for the precise definition of $A(X)$ and Fano Bott equivalence. 
Note that those three conditions are equivalent to the fourth condition (4). See Remark~\ref{rem:(4)}.

As an immediate corollary of Theorem~\ref{thm:FanoBott}, we conclude the following: 
\begin{cor}\label{cor:FanoBott}
Fano Bott manifolds are cohomologically rigid. 
\end{cor}
%Namely, it holds that 
%for two Fano Bott manifolds $X$ and $X'$, if there is an isomorphism between their integral cohomology rings, 
%then $X$ and $X'$ are diffeomorphic. 
Note that in the case $d \leq 4$, the cohomological rigidity for $d$-stage Bott manifolds is already known by \cite[Theorem 3.3]{Choi15} and \cite[Theorem 3.3]{Choi15} as well as \cite[Theorem 1.1]{HKM}. 

The key idea of the proof of Theorem~\ref{thm:FanoBott} relies on the identification of Fano Bott manifolds with \textit{rooted signed forests} (see Section~\ref{sec:signedrootedforest}).

\subsection{Structure of the paper} 
An overview of the present paper is as follows. 
In Section~\ref{sec:pre}, we recall the theory of toric varieties (e.g. primitive collections and primitive relations) 
and the description of cohomology rings of toric manifolds. We also introduce the invariants (s.v.e. and maximal basis number) on cohomology rings. 
We also give a sufficient condition (Lemma~\ref{key}) for two Bott manifolds to be diffeomorphic. 
In Section~\ref{sec:inv}, we introduce the upper triangular matrix $A(X)$ arising from a Fano Bott manifold $X$ and describe the cohomology ring $H^*(X)$ of $X$ in terms of $A(X)$. 
We also discuss s.v.e. and maximal basis number of $H^*(X)$ (Lemma~\ref{sve}). 
In Section~\ref{sec:signedrootedforest}, we associate signed rooted forests $T_X$ from Fano Bott manifolds $X$ 
and we prove a key proposition for the proof of Theorem~\ref{thm:FanoBott} (Proposition~\ref{prop:tree_cohomology}). 
In Section~\ref{sec:matrix}, we introduce three operations on the upper triangular matrices associated to Fano Bott manifolds and the notion of Fano Bott equivalence. 
We also see that those operations correspond to certain operations on the signs of some edges of the signed rooted forest (Proposition~\ref{prop:FanoBottMatrix}). 
Finally, in Section~\ref{sec:proof}, after preparing some more lemmas, we give a proof of Theorem~\ref{thm:FanoBott}.

\subsection*{Acknowledgements}
The authors would like to thank Mikiya Masuda for fruitful discussions about cohomological rigidity problems for toric manifolds. 
He also gave a lot of helpful comments on the first version of the present paper. 
A. Higashitani was supported in part by JSPS Grant-in-Aid for Scientific Research (C) 20K03513.

\bigskip

%%%%%%%%%%%%%%%%%%%%%%%%%%%%%%%%%%%%%%%%%%%%%%%%%%%%%%%%%%%%%%%%%%%%%%%%%%%%%%%%%%%%%%%%%%%%%%%%%%%%%%%%%
%%%%%%%%%%%%%%%%%%%%%%%%%%%%%%%%%%%%%%%%%%%%%%%%%%%%%%%%%%%%%%%%%%%%%%%%%%%%%%%%%%%%%%%%%%%%%%%%%%%%%%%%%
%%%%%%%%%%%%%%%%%%%%%%%%%%%%%%%%%%%%%%%%%%%%%%%%%%%%%%%%%%%%%%%%%%%%%%%%%%%%%%%%%%%%%%%%%%%%%%%%%%%%%%%%%
%%%%%%%%%%%%%%%%%%%%%%%%%%%%%%%%%%%%%%%%%%%%%%%%%%%%%%%%%%%%%%%%%%%%%%%%%%%%%%%%%%%%%%%%%%%%%%%%%%%%%%%%%
\section{Preliminaries}\label{sec:pre}

In this section, we recall the well-known description of the cohomology rings of toric manifolds. 
For our discussions, we recall the notions, \textit{primitive collections} and \textit{primitive relations}, which represent the toric variety in terms of linear relations of primitive ray vectors of the associated fan of a toric variety. 
We also introduce the invariants of cohomology rings to distinguish cohomology rings. 
At last, we give a sufficient condition for two toric manifolds to be diffeomorphic (Lemma~\ref{key}). 

\subsection{Complete nonsingular fans and toric manifolds} 

Please consult e.g. \cite{Fulton} or \cite{Oda} for the introduction to toric varieties and the associated fans. 

First, we recall the notion of complete nonsingular fans and their underlying simplicial complexes. 
A {\em fan} of dimension $d$ is a collection $\Sigma$ of rational polyhedral pointed cones in $\RR^d$ such that 
\begin{itemize}
\item[i)] each face of a cone $\sigma$ in $\Sigma$ also belongs to $\Sigma$, and 
\item[ii)] the intersection of two cones in $\Sigma$ is also a face of each of those cones. 
\end{itemize}
It is well known that the set of toric varieties of complex dimension $d$ up to algebraic isomorphism one-to-one corresponds to 
the set of fans of dimension $d$ up to unimodular equivalence. 
For a toric variety $X$, let $\Sigma_X$ be the corresponding fan. Similarly, for a fan $\Sigma$, let $X_\Sigma$ be the corresponding toric variety. 
It is known that a toric variety $X$ is compact if and only if $\Sigma_X$ is complete, i.e., $\bigcup_{\sigma \in \Sigma_X} \sigma = \RR^d$, 
and $X$ is smooth if and only if $\Sigma_X$ is nonsingular, i.e., the set of primitive ray generators of every maximal cone in $\Sigma_X$ forms a $\ZZ$-basis of $\ZZ^d$. 
Hence, the set of toric manifolds (namely, compact smooth toric varieties) one-to-one corresponds to the set of complete nonsingular fans. 

For a fan $\Sigma$, let $\Sigma_1$ be the set of primitive ray generators of $1$-dimensional cones in $\Sigma$. 
Given a complete nonsingular fan $\Sigma$, since nonsingular fans are simplicial, 
we see that $\Sigma$ has a structure of an abstract simplicial complex. 
Let $\Sigma_1=\{\vb_1,\ldots,\vb_m\}$. We define $$\calK(\Sigma)=\{I \subset [m] : \mathrm{cone}(\vb_i : i \in I) \in \Sigma\},$$ 
where $[m]=\{1,\ldots,m\}$. We call $\calK(\Sigma)$ the \textit{underlying simplicial complex} of $\Sigma$.

\subsection{Primitive collections and primitive relations}\label{sec:prim}
Next, we recall what primitive collections and primitive relations are.

\begin{defi}[{Primitive collection, \cite[Definition 2.6]{Bat91}}]
For a complete nonsingular fan $\Sigma$, we call a nonempty subset $\mathcal{P}=\{\xb_1,\ldots,\xb_k\} \subset \Sigma_1$ of $\Sigma_1$ 
a \textit{primitive collection} of $\Sigma$ if for each generator $\xb_i \in \mathcal{P}$ 
the elements of $\mathcal{P} \setminus \{\xb_i\}$ generate a $(k-1)$-dimensional cone in $\Sigma$, 
while $\mathcal{P}$ does not generate any $k$-dimensional cone in $\Sigma$. 
In other words, $\calP$ is a \textit{minimal non-face} of $\calK(\Sigma)$. Let $\PC(\Sigma)$ be the set of all primitive collections of $\Sigma$. 
\end{defi}
\begin{defi}[{Primitive relation, \cite[Definition 2.8]{Bat91}}]
For a complete nonsingular fan $\Sigma$, let $\mathcal{P}=\{\xb_1,\ldots,\xb_k\}$ be a primitive collection of $\Sigma$. 
Let $\sigma$ be the cone in $\Sigma$ of the smallest dimension containing $\xb_1+\cdots+\xb_k$ 
and let $\yb_1,\ldots,\yb_m \in \Sigma_1$ be the minimal system of primitive ray generators of $\sigma$. 
Then there exists a unique linear combination $n_1\yb_1+\cdots+n_m\yb_m$ with positive integer coefficients $n_i$ which is equal to $\xb_1+\cdots+\xb_k$. 
We call the linear relation \begin{align}\label{prim_rel} \xb_1+\cdots+\xb_k=n_1\yb_1+\cdots+n_m\yb_m \end{align} 
the \textit{primitive relation} associated with $\mathcal{P}$.

Given a primitive collection $\calP$ with its associated primitive relation \eqref{prim_rel}, 
let $\deg(\mathcal{P}) = k-(n_1+\cdots+n_m)$, which we call the \textit{degree} of $\mathcal{P}$. 
\end{defi}

We can verify whether the toric manifold $X$ is Fano or not by the degrees of primitive relations of $\Sigma_X$. 
\begin{prop}[{\cite[Proposition 2.3.6]{Bat99}}]\label{prop:Fano}
A toric manifold $X$ is Fano if and only if we have $\deg(\mathcal{P})>0$ for every primitive collection $\mathcal{P}$ of $\Sigma_X$. 
\end{prop}
Given two complete nonsingular fans $\Sigma$ and $\Sigma'$, 
we say that $\PC(\Sigma)$ and $\PC(\Sigma')$ are isomorphic if there is a bijection between $\Sigma_1$ and $\Sigma_1'$ 
which induces a bijection between not only $\PC(\Sigma)$ and $\PC(\Sigma')$ but also their primitive relations. 
%The equivalence of $\PC(\Sigma)$ corresponds to that of toric Fano manifolds as algebraic varieties. 
\begin{prop}[{\cite[Proposition 2.1.8 and Theorem 2.2.4]{Bat99}}]\label{prop:iso}
Two toric Fano manifolds $X$ and $X'$ are isomorphic as algebraic varieties if and only if $\PC(\Sigma_X)$ and $\PC(\Sigma_{X'})$ are isomorphic. 
\end{prop}

\subsection{Cohomology rings of toric manifolds and their invariants}

Next, we recall the description of the cohomology rings of toric manifolds. 

\begin{prop}[{\cite[Theorem 5.3.1]{ToricTopology}}]\label{prop:compute}
Let $X$ be a toric manifold of complex dimension $d$, let $\Sigma=\Sigma_X$ and let $\Sigma_1=\{{\bf v}_1,\ldots,{\bf v}_m\} \subset \RR^d$. 
Then the cohomology ring of $X$ can be described as follows: 
	$$H^*(X) \cong \ZZ[x_1,\ldots,x_m]/(I_X + J_X),$$
where 
\begin{align*}
I_X=\left(\prod_{i \in F} x_i : F \subset [m], \{\vb_i : i \in F\} \in \PC(\Sigma_X)\right) \text{ and }
J_X=\left(\sum_{j=1}^m {\bf v}_j^i x_j : i=1,\ldots,d \right) 
\end{align*}
and ${\bf v}_j^i$ denotes the $i$-th entry of ${\bf v}_j \in \RR^d$. 
\end{prop}

%%%%%%%%%%%%%%%%%%%%%%%%%%%%%%%%%%
%%%%%%%%%%%%%%%%%%%%%%%%%%%%%%%%%%
In order to distinguish cohomology rings up to isomorphism, we prepare the invariants on cohomology rings. 

%\begin{defi}[s.v.e., $k$-v.e.]
%Let $k\ge 2$ and $R=\ZZ$ or $\ZZ/p\ZZ$ where $p$ is a prime number. 
%A nonzero linear form in $H^*(X) \otimes R$ is said to be \emph{ $k$-v.e.} over $R$ 
%if it is primitive and its $k$-th power vanishes in $H^*(X)\otimes R$. 
%When $R=\ZZ$, the word \lq\lq over $\ZZ$\rq\rq\ will be omitted. 
%We call $2$-v.e. \emph{s.v.e.} (square vanishing element). 
%\end{defi}
\begin{defi}[s.v.e., maximal basis number]
Let $X$ be a toric manifold. 

A nonzero linear form in $H^*(X)$ is said to be \emph{s.v.e.} (square vanishing element) 
if it is primitive and its second power vanishes in $H^*(X)$. 

Let $V$ be the set of all s.v.e. of the cohomology ring $H^*(X)$ of $X$. Define 
$$\calB = \{ S \subset V : \text{$S$ is a part of a $\ZZ$-basis of $H^*(X)$} \}.$$
Then there exists $S_{\rm max} \in \calB$ such that $|S| \leq |S_{\rm max}| \text { for any }S \in \calB$. 
We call a set $S_{\rm max}$ a \textit{maximal basis} of s.v.e. of $H^*(X)$ and $|S_{\rm max}|$ a \textit{maximal basis number} of $H^*(X)$. 
\end{defi}

Note that the number of s.v.e. and the maximal basis number are invariants of $H^*(X)$. 
We refer the readers to \cite[Examples 2.7 and 2.9]{HKM} for examples of s.v.e and maximal basis numbers of $H^*(X)$ and how to compute them. 
%%%%%%%%%%%%%%%%%%%%%%%%%%%%%%%%%%
%%%%%%%%%%%%%%%%%%%%%%%%%%%%%%%%%%

\subsection{Diffeomorphism lemma}

Finally, we recall the key lemma for the proof of Theorem~\ref{thm:FanoBott}. The following lemma directly follows from \cite[Lemma 2.3]{HKM}. 
\begin{lem}[cf. {\cite[Lemma 2.3]{HKM}}]\label{key}
Let $X,X'$ be $d$-stage Fano Bott manifolds and let $\Sigma,\Sigma'$ be the associated complete nonsingular fans, respectively. 
Let $\Sigma_1=\{{\bf v}_1,\ldots,{\bf v}_{2d}\}$ and let $\Sigma_1'=\{{\bf v}_1',\ldots,{\bf v}_{2d}'\}$. 
If ${\bf v}_i=\pm {\bf v}_i'$ for each $i=1,\ldots,2d$ by reordering ${\bf v}_i$'s if necessary such that $\calK(\Sigma)$ is unchanged, then $X$ and $X'$ are diffeomorphic. 
\end{lem}

%%%%%%%%%%%%%%%%%%%%%%%%%%%%%%%%%%%%%%%%%%%%%%%%%%%%%%%%%%%%%%%%%%%%%%%%%%%%%%%%%%%%%%%%%%%%%%%%%%%%%%%%%%%%%%%%%%%%%%%%%%%%%%%%%%%%
%%%%%%%%%%%%%%%%%%%%%%%%%%%%%%%%%%%%%%%%%%%%%%%%%%%%%%%%%%%%%%%%%%%%%%%%%%%%%%%%%%%%%%%%%%%%%%%%%%%%%%%%%%%%%%%%%%%%%%%%%%%%%%%%%%%%
%%%%%%%%%%%%%%%%%%%%%%%%%%%%%%%%%%%%%%%%%%%%%%%%%%%%%%%%%%%%%%%%%%%%%%%%%%%%%%%%%%%%%%%%%%%%%%%%%%%%%%%%%%%%%%%%%%%%%%%%%%%%%%%%%%%%
%%%%%%%%%%%%%%%%%%%%%%%%%%%%%%%%%%%%%%%%%%%%%%%%%%%%%%%%%%%%%%%%%%%%%%%%%%%%%%%%%%%%%%%%%%%%%%%%%%%%%%%%%%%%%%%%%%%%%%%%%%%%%%%%%%%%
%%%%%%%%%%%%%%%%%%%%%%%%%%%%%%%%%%%%%%%%%%%%%%%%%%%%%%%%%%%%%%%%%%%%%%%%%%%%%%%%%%%%%%%%%%%%%%%%%%%%%%%%%%%%%%%%%%%%%%%%%%%%%%%%%%%%
%%%%%%%%%%%%%%%%%%%%%%%%%%%%%%%%%%%%%%%%%%%%%%%%%%%%%%%%%%%%%%%%%%%%%%%%%%%%%%%%%%%%%%%%%%%%%%%%%%%%%%%%%%%%%%%%%%%%%%%%%%%%%%%%%%%%
%%%%%%%%%%%%%%%%%%%%%%%%%%%%%%%%%%%%%%%%%%%%%%%%%%%%%%%%%%%%%%%%%%%%%%%%%%%%%%%%%%%%%%%%%%%%%%%%%%%%%%%%%%%%%%%%%%%%%%%%%%%%%%%%%%%%
%%%%%%%%%%%%%%%%%%%%%%%%%%%%%%%%%%%%%%%%%%%%%%%%%%%%%%%%%%%%%%%%%%%%%%%%%%%%%%%%%%%%%%%%%%%%%%%%%%%%%%%%%%%%%%%%%%%%%%%%%%%%%%%%%%%%

\bigskip

\section{Fano Bott manifolds and their cohomology rings}\label{sec:inv}

In this section, we introduce upper triangular matrices $(n_{ij})$ associated to Fano Bott manifolds. 
The cohomology rings of Fano Bott manifolds can be described by using such matrices (see \eqref{eq:ideal}). 
We discuss the s.v.e. of their cohomology rings (Lemma~\ref{sve}). 
We refer the reader to \cite[Section 7.8]{ToricTopology} for the introduction to Bott manifolds.

\subsection{Upper triangular matrices associated to Fano Bott manifolds} 
Let $X$ be a $d$-stage Fano Bott manifold and let $\Sigma=\Sigma_X$. Then $|\Sigma_1|=2d$. Let $\Sigma_1=\{{\bf v}_i^\pm : i=1,\ldots,d\}$. 
In general, the underlying simplicial complex of the associated fan of any Bott manifold is the boundary complex of the cross-polytope of dimension $d$. 
Thus, the primitive collections look like $\{{\bf v}_i^+,{\bf v}_i^-\}$ for each $i$. 
Moreover, by Proposition~\ref{prop:Fano}, the associated primitive relations of $\PC(\Sigma)$ look as follows: 
\begin{equation}\label{eqeq}
\begin{split}
		{\bf v}_1^+ + {\bf v}_1^- &= {\bf v}_{\varphi(1)}^{\sigma(1)}, \\ 
		{\bf v}_2^+ + {\bf v}_2^- &= {\bf v}_{\varphi(2)}^{\sigma(2)}, \\ 
		              &\;\vdots \\
		{\bf v}_d^+ + {\bf v}_d^- &= {\bf v}_{\varphi(d)}^{\sigma(d)}, 
\end{split}
\end{equation}
where we let ${\bf v}_{d+1}^\pm={\bf 0}$, $\varphi$ is a map $\varphi : [d] \rightarrow [d+1] \setminus\{1\}$ satisfying that there is $k_i$ with $\varphi^{k_i}(i)=d+1$ for each $i \in [d]$ 
and $\sigma$ is a map $\sigma : [d] \rightarrow \{\pm\}$. 
Note that the condition ``there is $k_i$ with $\varphi^{k_i}(i)=d+1$ for each $i \in [d]$'' is derived from the linear independence of ${\bf v}_1^{\epsilon_1},\ldots,{\bf v}_d^{\epsilon_d}$ for any choice of $\epsilon_i \in \{\pm\}$. 
Up to the equivalence of $\PC(\Sigma)$, we may assume that 
\begin{align}\label{condition}
\text{$\varphi$ satisfies that $i < \varphi(i) \leq d+1$ for each $i$.} 
\end{align}
In what follows, we will always assume \eqref{condition}.

In general, $d$-stage Bott manifolds are determined by a collection of integers $(n_{ij})_{1 \leq i < j \leq d}$ arranged in an upper triangular matrix form (see \cite[Section 7.8]{ToricTopology}). 
In the case of Fano Bott manifolds, we can construct this upper triangular matrix $(n_{ij})$ from the primitive relations as follows. 
Since we have ${\bf v}_i^- = -{\bf v}_i^+ + {\bf v}_{\varphi(i)}^{\sigma(i)}$ for each $i$, by the condition \eqref{condition}, we have 
$$
{\bf v}_i^+ + {\bf v}_i^- = -{\bf v}_{\varphi(i)}^+ -{\bf v}_{\varphi^2(i)}^+ - \cdots -{\bf v}_{\varphi^{k_i-1}(i)}^+ +{\bf v}_{\varphi^{k_i}(i)}^+.
$$
Thus, we can rewrite the primitive relations \eqref{eqeq} as follows: 
\begin{table}[htb]
	\centering
	\begin{tabular}{ccccccccccc}
		${\bf v}_1^+ + {\bf v}_1^-$ &=& $n_{12}{\bf v}_2^+$ &$+$& $n_{13}{\bf v}_3^+$ &$+$& $n_{14}{\bf v}_4^+$ &$+$& $\cdots$ &$+$& $n_{1d}{\bf v}_d^+$, \\ 
		${\bf v}_2^+ + {\bf v}_2^-$ &=&                        &      & $n_{23}{\bf v}_3^+$ &$+$& $n_{24}{\bf v}_4^+$ &$+$& $\cdots$ &$+$& $n_{2d}{\bf v}_d^+$, \\ 
		${\bf v}_3^+ + {\bf v}_3^-$ &=&                        &      &                          &      & $n_{34}{\bf v}_4^+$ &$+$& $\cdots$ &$+$& $n_{3d}{\bf v}_d^+$, \\ 
		$\vdots$        &$\vdots$&                 &      &  &  &                        &      &               &      & $\vdots$            \\ 
		${\bf v}_{d-1}^+ + {\bf v}_{d-1}^-$ &=&               &      &                          &      &                          &      &               &      & $n_{d-1,d}{\bf v}_d^+$, \\ 
		${\bf v}_d^+ + {\bf v}_d^-$ &=&                       &      &                          &      &                          &      &               &      & ${\bf 0}$, 
	\end{tabular}
\end{table}
${}$\\
where $n_{ij} \in \{0,\pm 1\}$.

By \cite{Suyama}, Fano Bott manifolds are characterized in terms of $(n_{ij})_{1 \leq i<j \leq d}$. 
\begin{thm}[{\cite[Theorem 8]{Suyama}}]\label{suyama}
	Let $X$ be a $d$-stage Bott manifold and let $(n_{ij})_{1 \leq i<j \leq d}$ be the associated upper triangular matrix. 
	Then $X$ is Fano if and only if for any $p = 1, \ldots, d-1$, one of the following conditions holds: 
	\begin{align*}
	&(1) ~~ n_{p,p+1} = \cdots = n_{pd} = 0, \\
	&(2) ~~ n_{pr} = 
	\begin{cases}
	1 ~~ &(r = q)\\
	0 ~~ &(r \neq q)
	\end{cases} ~~\text{ for }r=p+1,\ldots,d, \\
	&(3) ~~ n_{pr} = 
	\begin{cases}
	0 ~~ &(r=p+1,\ldots,q-1), \\
	-1 ~~ &(r = q), \\
	n_{qr} ~~ &(r=q+1,\ldots,d). \\
	\end{cases} ~~
	\end{align*}
\end{thm}

Let $R^{d \times d}$ denote the set of all $d \times d$ matrices whose entries are in the set $R$. 
\begin{defi}\label{def:FB(d)}
Given a $d$-stage Fano Bott manifold $X$, let $$A(X)=(n_{ij})_{1 \leq i,j \leq d} \in \{0,\pm 1\}^{d \times d}$$ be the upper triangular matrix 
whose upper triangle part is equal to $n_{ij}$ defined as above and whose remaining parts are all $0$. 
Moreover, let %$$\FB(d)=\{M \in \mat_d(\ZZ) : M \text{ satisfies the conditions in Theorem~\ref{suyama}}\}.$$ 
$$\FB(d)=\{A(X) \in \{0,\pm1\}^{d \times d} : X \text{ is a $d$-stage Fano Bott manifold}\}.$$ 
Namely, $\FB(d)$ consists of all upper triangular matrices satisfying the condition in Theorem~\ref{suyama}. 
\end{defi}

%For the proof of Theorem~\ref{thm:FanoBott}, we prepare some lemmas in this section. 
%The key for the proof is the nice correspondence between (primitive relations of) Fano Bott manifolds and signed rooted forests (see Section~\ref{sec:signedrootedforest}). 
\subsection{Cohomology rings of Fano Bott manifolds and their s.v.e.}
For a $d$-stage Fano Bott manifold $X$, by using the upper triangular matrix $(n_{ij})$ associated to $X$, 
the cohomology ring $H^*(X)$ can be written as follows (see, e.g., \cite[(2.5)]{CLMP}): 
\begin{equation}\label{eq:ideal}\begin{split}
&H^*(X) \cong \ZZ[x_1,\ldots,x_d]/\calI, \text{ where} \\
&\calI = (f_i : i=1,\ldots,d) \;\text{and}\\ 
&f_i =x_i^2-(n_{1i}x_1+\cdots +n_{i-1,i}x_{i-1})x_i\text{ for }i=1,\ldots,d.
\end{split}\end{equation}
For the analysis of s.v.e. of $H^*(X)$, we perform the following computations. Let $a_i \in \ZZ$. Then we have 
\begin{align*}
(a_1x_1 + \cdots + a_dx_d)^2 &= \sum_{i=1}^d a_i^2 x_i^2 + \sum_{1 \leq i < j \leq d}2a_ia_jx_ix_j \\
&=\sum_{i=1}^da_i^2f_i+\sum_{j=1}^da_j^2(n_{1j}x_1+\cdots +n_{j-1,j}x_{j-1})x_j+ \sum_{1 \leq i < j \leq d}2a_ia_jx_ix_j \\
&=\sum_{i=1}^da_i^2f_i+\sum_{1 \leq i < j \leq d}a_j(a_jn_{ij}+2a_i)x_ix_j. 
\end{align*}
Hence, %$\NF((a_1x_1 + \cdots + a_mx_m)^2)=\sum_{1 \leq i < j \leq m}a_j(a_jn_{j,i}+2a_i)x_ix_j$ and 
we obtain the following: 
\begin{equation}\label{eq:NF}
\begin{split}
(a_1x_1 + \cdots + a_dx_d)^2=0 &\text{ in }H^*(X) \\ \Longleftrightarrow \; 
&a_2=0 ~~\text{or}~~ n_{12}a_2=-2a_1, \text{ and}\\
&a_3=0 ~~\text{or}~~ n_{13}a_3=-2a_1 ~~\text{or}~~ n_{23}a_3=-2a_2, \text{ and}\\
&\cdots\cdots, \\
&a_d=0 ~~\text{or}~~ n_{1d}a_d=-2a_1 ~~\text{or}~~ \cdots ~~\text{or}~~ n_{d-1,d}a_d=-2a_{d-1}. 
\end{split}
\end{equation}
Let $\mathcal{S}(H^*(X))$ be the set of all s.v.e. of $H^*(X)$. Then we see the following: 
\begin{lem}\label{sve}
	Work with the same notation as above. Then $\mathcal{S}(H^*(X))$ can be divided into three disjoint subsets 
	$\{g_i\}_{i \in I}, \{g'_i\}_{i \in I}$ and $\{h_j\}_{j \in J}$ with some $I,J \subset [d]$ satisfying the following: 
\begin{itemize}
\item[(1)] $\{g_i\}_{i \in I'} \cup \{g'_i\}_{i \in I \setminus I'} \cup \{h_j\}_{j \in J}$ is a maximal basis of $\calS(H^*(X))$ for any (possibly empty) $I' \subset I$; 
\item[(2)] $H^*(X) / (g_i) \cong H^*(X) / (g'_i)$  for each $i$. 
\end{itemize}
\end{lem}
\begin{proof}
	\noindent
        {\bf (The first step)}: Let $g=a_1x_1+\cdots+a_dx_d \in \calS(H^*(X))$. 
        First, we determine what kind of $g$ appears. %i.e., $g \in \calS(H^*(X))$. 

        \noindent
        \underline{The case $a_p \neq 0$ and $a_i = 0$ for any $i$ with $i \neq p$}: Then $g=x_p$ and $g^2=x_p^2=0$. 
        
        \noindent
        \underline{The case $a_p \neq 0$, $a_q \neq 0$ and $a_i = 0$ for any $i$ with $i \neq p,q$}: 
        Let $p<q$. Then $g^2=0$ implies that $n_{pq}a_q = -2a_p \neq 0$. Thus, $n_{pq} = \pm 1$, i.e., $a_q = \pm 2a_p$. 
        Moreover, $$n_{iq}a_q = -2a_i = 0 ~~(1 \leq i \neq p < q) ~~\Rightarrow~~ n_{iq} = 0 ~~(1 \leq i \neq p < q). $$
	Hence, we may assume that $n_{pq} = 1$, i.e., $a_q = -2a_p$. Namely, we have $(x_p-2x_q)^2=0$. Moreover, in this case, we have 
	$$n_{ip}a_p = -2a_i = 0 ~~(1 \leq i < p) ~~\Rightarrow~~ n_{ip} = 0 ~~(1 \leq i < p),$$
	so we obtain that $x_p^2=0$. 
        
        \noindent
        \underline{The case $a_p \neq 0$, $a_q \neq 0$ and $a_r \neq 0$ with $p<q<r$}: 
        Then it follows from \eqref{eq:NF} that 
%	\begin{align*}
%	\begin{cases}
%	n_{pq}a_q = -2a_p \neq 0, \\ n_{pr}a_r = -2a_p \neq 0, \\ n_{qr}a_r = -2a_q \neq 0. 
%	\end{cases} 
%        \end{align*}
	$$n_{pq}a_q = -2a_p \neq 0, \;\; n_{pr}a_r = -2a_p \neq 0 \text { and } n_{qr}a_r = -2a_q \neq 0. $$
	Thus, $n_{pq}, n_{pr}, n_{qr} \in \{\pm 1\}$. However, in this case, we have 
        $\begin{cases}
	a_r = \pm a_q \\ a_r = \pm 2a_q, 
	\end{cases}$
	implying that $a_q=a_r=0$, a contradiction. 
        
        Therefore, we see that $\calS(H^*(X)) \subset \{x_p : 1 \leq p \leq d \} \cup \{x_p-2x_q : 1 \leq p<q \leq d\}$. 
        
	\noindent
	{\bf (The second step)}: Let $(x_{p_1}-2x_{q_1})^2=0$ and $(x_{p_2}-2x_{q_2})^2=0$. 
        We claim that $\{p_1, q_1\} \neq \{p_2, q_2\}$ implies $\{p_1, q_1\} \cap \{p_2, q_2\} = \emptyset$. 
	By the above discussions, we may assume that $n_{p_1q_1} = 1$ and $n_{p_2q_2} = 1$, and we also have 
	\begin{align*}
	n_{iq_1} = 0 ~~(1 \leq i < q_1, i \neq p_1), ~n_{ip_1} = 0 ~~(1 \leq i < p_1), \\
	n_{iq_2} = 0 ~~(1 \leq i < q_2, i \neq p_2), ~n_{ip_2} = 0 ~~(1 \leq i < p_2). 
	\end{align*}
	Then $p_1=q_2$ and $p_2=q_1$ never happen. If $q_1=q_2$, then $p_1=p_2$. 
	If $p_1=p_2$, Theorem~\ref{suyama} (2) implies that $q_1=q_2$ since we assume $n_{p_1q_1}=n_{p_2q_2}=1$. 
        Hence, $\{p_1, q_1\} \cap \{p_2, q_2\} = \emptyset$ holds if $\{p_1, q_1\} \neq \{p_2, q_2\}$．

	\noindent
	{\bf (The third step)}: 
        Now, we consider the following subsets of $\calS(H^*(X))$: 
	\begin{align*}
	& \{g_i\}_i = \{ x_p ~:~ (x_p-2x_q)^2=0 \text{ for some }q\}, \;\; \{h_j\}_j = \{ x_p ~:~ x_p^2=0 \} \setminus \{g_i\}_i, \\
	& \{g'_i\}_i = \{ x_p-2x_q ~:~ (x_p-2x_q)^2=0 \}. 
	\end{align*}
	Since $x_p \in \{g_i\}_i$ if and only if $x_p-2x_q \in \{g_i'\}_i$ (see the first step), 
        we can simultaneously index both of the sets $\{g_i\}$ and $\{g_i'\}$ by $I \subset [d]$. 
        Moreover, it follows from the second step that the condition (1) holds. 
        
        We claim that the condition (2) also holds. More precisely, we prove the ring isomorphism $H^*(X) / (x_p) \cong H^*(X) / (x_p - 2x_q)$. 
	We may assume that $n_{pq} = 1$. By Theorem~\ref{suyama} (2), we have $n_{pi} = 0 ~~(p<i)$, and by $n_{iq} = 0 ~~(1 \leq i < q, i \neq p)$, 
	we see that a generator of $\calI$ which is divisible by $x_p$ is either $x_p^2$ or $x_q(x_q-x_p)$. 
%	Therefore, the ring homomorphism $F : H^*(M) / (x_p) \rightarrow H^*(M) / (x_p - 2x_q)$ defined by 
%	\begin{align*}
%	F(x_i) = x_i ~(i \neq p), ~~~F(x_p) = x_p - 2x_q
%	\end{align*}
%        becomes an isomorphism. 
	Therefore, the ring endomorphism $F : H^*(X) \rightarrow H^*(X)$ defined by 
	\begin{align*}
	F(x_i) = x_i ~(i \neq p), ~~~F(x_p) = -x_p + 2x_q
	\end{align*}
        is well-defined and induces an isomorphism between $H^*(X)/(x_p)$ and $H^*(X)/(x_p-2x_q)$. 
\end{proof}

\begin{rem}\label{rem:sve}
For a Fano Bott manifold, a maximal basis of $\calS(H^*(X))$ is not unique as claimed in the condition (1) of Lemma~\ref{sve}. 
However, the condition (2) of Lemma~\ref{sve} says that for any maximal basis $S_{\rm max}$ of $\calS(H^*(X))$, 
$\{H^*(X)/(f) : f \in S_{\rm max}\}$ is unique up to ring isomorphisms. 
Therefore, when we consider the family of quotient rings by s.v.e., we may only consider s.v.e. of the form $x_p$, 
namely, we do not need to treat s.v.e. of the form $x_p-2x_q$. 
\end{rem}

\bigskip

\section{Correspondence between Fano Bott manifolds and signed rooted forests}\label{sec:signedrootedforest}
In this section, we introduce the notion of signed rooted forests. %and we give a kind of a characterization of rooted trees (Lemma~\ref{treeiso}). 
We construct the signed rooted forest $T_{\varphi,\sigma}$ from a map $\varphi : [d] \rightarrow [d+1] \setminus \{1\}$ with \eqref{condition} and a map $\sigma : [d] \rightarrow \{\pm\}$. 
Namely, we can construct a signed rooted forest $T_X$ from a Fano Bott manifold $X$. 
By using this, we establish a correspondence between Fano Bott manifolds and signed rooted forests $T_{\varphi,\sigma}$. 
We also observe that an operation on a signed rooted forest, called \textit{leaf-cutting} by $v_\alpha$, corresponds to taking a quotient by a certain element $x_\alpha$.

We call $T=(V,E,v_0)$ a \textit{rooted tree} on the vertex set $V$ with its root $v_0$ and the edge set $E$
if $(V,E)$ is a tree, which is a connected graph having no cycle, and a vertex $v_0 \in V$, called a \textit{root}, is fixed. 
A \textit{rooted forest} is a disjoint union of rooted trees. We recall some notions on a rooted forest $T$: 
\begin{itemize}
\item A vertex $v$ (resp. $v'$) in $T$ is called a \textit{child} (resp. \textit{parent}) of a vertex $v'$ (resp. $v$) 
if $\{v',v\}$ is an edge of a connected component $T_0$ of $T$ and $v'$ is closer to the root of $T_0$ than $v$. 
\item A \textit{descendant} of a vertex $v$ in $T$ means any vertex which is either a child of $v$ or recursively a descendant of $v$. 
\item We call a vertex $v$ in $T$ a \textit{leaf} if no child is adjacent to $v$. 
%\item Given a vertex $v$ in $T$, we say that a subtree $T_0$ of $T$ is the \textit{rooted subtree of $T$ induced by $v$} 
%if $T_0$ is the rooted tree with its root $v$ consisting of all descendants of $v$. 
\end{itemize}
We call a rooted forest \textit{signed} if $+$ or $-$ is assigned to each of its edges. 
Given an edge $e$, we denote the sign of $e$ by $\sign(e)$. 

Let $T$ and $T'$ be two rooted forests. We say that $T$ and $T'$ are \textit{isomorphic as rooted forests} 
if there is a bijection $f:V(T) \rightarrow V(T')$ between the sets of vertices which induces a bijection between the roots and between the sets of edges, 
and we call $f$ an isomorphism as \textit{signed} rooted forests if $f$ also preserves all signs of the edges.

\begin{defi}[Signed rooted forests associated to Fano Bott manifolds]
Let $\varphi : [d] \rightarrow [d+1] \setminus \{1\}$ be a map with \eqref{condition} and let $\sigma : [d] \rightarrow \{\pm\}$. 
We define the signed rooted forest $T_{\varphi,\sigma}$ from $\varphi$ and $\sigma$ as follows: 
\begin{align*}
&V(T_{\varphi,\sigma}) = \{ v_1, \ldots , v_d \}, \\ 
&E(T_{\varphi,\sigma}) = \{ \{v_{\varphi(i)}, v_i\} : 1 \leq i \leq d, \; \varphi(i) \neq d+1\} \text{ with }\sign(\{v_{\varphi(i)}, v_i\})=\sigma(i) \text{ for each $i$,}\\
&\text{the roots are $v_i$'s with $\varphi(i)=d+1$}. 
\end{align*}
As mentioned in Section~\ref{sec:inv}, we can associate the above maps $\varphi$ and $\sigma$ from Fano Bott manifolds, 
so we can construct $T_{\varphi,\sigma}$ from Fano Bott manifolds. Let $T_X=T_{\varphi,\sigma}$ be such signed rooted forest. 
\end{defi}

\begin{rem}\label{rem:tree_Bott}
Given a Fano Bott manifold $X$, we see that the signed rooted forest $T_{\varphi,\sigma}$ associated to $X$ constructed in the above way is well-defined. 
In fact, for two Fano Bott manifolds $X$ and $X'$ and their associated signed rooted forests $T_X$ and $T_{X'}$, %let $\Sigma=\Sigma_X$ and $\Sigma'=\Sigma_{X'}$. 
it is straightforward to check that the equivalence between $\PC(\Sigma_X)$ and $\PC(\Sigma_{X'})$, which is a bijection between $({\Sigma_X})_1$ and $({\Sigma_{X'}})_1$, 
directly implies the isomorphism between two corresponding signed rooted forests $T_X$ and $T_{X'}$, which is a bijection between $V(T_X)$ and $V(T_{X'})$. 

On the other hand, given a signed rooted forest $T$, we can reconstruct the primitive relations associated to $T$, 
so we can associate a Fano Bott manifold from $T$. Let $X_T$ be the Fano Bott manifold associated to a signed rooted forest $T$. 
\end{rem}
%In what follows, we use the notation $T_X$ which is the signed rooted forest arising from a Fano Bott manifold $X$. 

\begin{ex}\label{ex:rooted_forest}
(1) Let us consider the Fano Bott manifold $X$ associated to the following primitive relations: 
\begin{align*}
&{\bf v}_1^+ + {\bf v}_1^- = {\bf v}_2^+, \\
&{\bf v}_2^+ + {\bf v}_2^- = {\bf v}_5^-, \\
&{\bf v}_3^+ + {\bf v}_3^- = {\bf v}_4^-, \\
&{\bf v}_4^+ + {\bf v}_4^- = {\bf v}_5^+, \\
&{\bf v}_5^+ + {\bf v}_5^- = {\bf 0}. 
\end{align*}
Then the associated signed rooted forest $T_X$ looks as follows. 
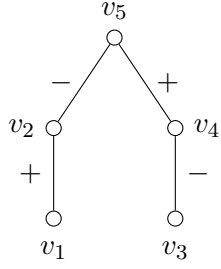
\begin{center}
\begin{tikzpicture}[]
\node[level 1, label=above:{$v_5$}] {}
child {
	node[level 2, label=left:{$v_2$}] {}
	child {
		node[level 3, label=below:{$v_1$}] {}
		edge from parent node[left]{$+$}
	}
	edge from parent node[left]{$-$}
}
child {
	node[level 2, label=right:{$v_4$}] {}
	child {
		node[level 3, label=below:{$v_3$}] {}
		edge from parent node[right]{$-$}
	}
	edge from parent node[right]{$+$}
};
\end{tikzpicture}
\end{center}

\noindent
(2) Let us consider the Fano Bott manifold $X$ associated to the following primitive relations: 
\begin{align*}
&{\bf v}_1^+ + {\bf v}_1^- = {\bf v}_3^+, \\
&{\bf v}_2^+ + {\bf v}_2^- = {\bf v}_3^-, \\
&{\bf v}_3^+ + {\bf v}_3^- = {\bf 0}, \\
&{\bf v}_4^+ + {\bf v}_4^- = {\bf v}_5^+, \\
&{\bf v}_5^+ + {\bf v}_5^- = {\bf 0}. 
\end{align*}
Then the associated signed rooted forest $T_X$ looks as follows. 

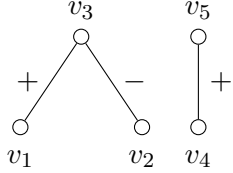
\begin{center}
\begin{tikzpicture}[]
\node[level 1, label=above:{$v_3$}] {}
child {
	node[level 2, label=below:{$v_1$}] {}
	edge from parent node[left]{$+$}
}
child {
	node[level 2, label=below:{$v_2$}] {}
	edge from parent node[right]{$-$}
};
\end{tikzpicture}
\begin{tikzpicture}[]
\node[level 1, label=above:{$v_5$}] {}
child {
	node[level 2, label=below:{$v_4$}] {}
	edge from parent node[right]{$+$}
};
\end{tikzpicture}
\end{center}
\end{ex}

\begin{rem}\label{rem:forest_matrix}
Let $X$ be a $d$-stage Fano Bott manifold and let $A(X)=(n_{ij})$ be the upper triangular matrix associated to $X$. 
Then Theorem~\ref{suyama} claims that $A(X)$ is completely determined by the left-most nonzero entry of each row. 
Moreover, the left-most nonzero entry of each row, which is $1$ or $-1$, 
one-to-one corresponds to the signed edge of the associated signed rooted forest $T_X$. Namely, we see that 
\begin{equation}\label{eq:signed_edge}
\begin{split}
&n_{i,i+1}=\cdots=n_{i,j-1}=0, \; n_{ij}=1 \text{ (resp. $n_{ij}=-1$)} \\
&\quad\quad\Longleftrightarrow\;\; \{v_j,v_i\} \in E(T_X) \text{ and }\sign(\{v_j,v_i\})=+ \text{ (resp. $\sign(\{v_j,v_i\})=-$)}
\end{split}
\end{equation}
Furthermore, we can read off all entries of $(n_{ij})$ from $T_X$ as follows: for $1 \leq i<j \leq d$, let 
\begin{align*}
n_{ij}=\begin{cases}
1, \;&\text{if }\{v_i,v_j\} \in E(T_X) \text{ and }\sign(\{v_i,v_j\})=+, \\
1, \;&\text{if there is an upward path $(v_{i_0},\ldots,v_{i_k})$ from $v_i=v_{i_0}$ to $v_j=v_{i_k}$ such that} \\
&\text{$\sign(\{v_{i_{\ell-1}},v_{i_\ell}\})=-$ for $\ell=1,\ldots,k-1$ and $\sign(\{v_{i_{k-1}},v_{i_k}\})=+$}, \\
-1, \;&\text{if there is an upward path from $v_i$ to $v_j$ all of whose edges have the sign $-$}, \\
0, \;&\text{otherwise}, 
\end{cases}
\end{align*}
where we call a path $(v_{i_0},v_{i_1},\ldots,v_{i_k})$ in the signed rooted forest \textit{upward} if $v_{i_j}$ is a child of $v_{i_{j+1}}$ for all $j=1,\ldots,k-1$. 
For Example~\ref{ex:rooted_forest} (1), we see that 
\begin{align*}
n_{12}=n_{35}=n_{45}=1, \; n_{25}=n_{34}=-1, \; n_{ij}=0 \text{ otherwise}. 
\end{align*}

Let $T_1,\ldots,T_k$ be the connected components of a signed rooted forest $T$. 
It then follows from the above construction that the upper triangular matrix associated to $T$ becomes a direct sum of those of $T_1,\ldots,T_k$. 
\end{rem}

Given a Fano Bott manifold $X$, we observe the following: 
\begin{itemize}
\item By definition of $T_{\varphi,\sigma}$, we see that $v_{\alpha}$ is a leaf of $T_{\varphi,\sigma}$ if and only if $\alpha \not\in \varphi([d])$. 
Thus, for a leaf $v_\alpha$ of $T_{\varphi,\sigma}$, we obtain that $n_{i\alpha} = 0$ for each $i=1,\ldots,\alpha-1$, i.e., $x_{\alpha}^2=0$ in $H^*(X)$. 
\item On the other hand, if $x_{\alpha}^2=0$ in $H^*(X)$, then $n_{1\alpha} = \cdots = n_{\alpha-1,\alpha} = 0$ by \eqref{eq:NF}. 
Thus, we obtain that $\alpha \not\in \varphi([d])$, i.e., $v_{\alpha}$ is a leaf of $T_X$. 
\end{itemize}
Therefore, the set of leaves of $T_X$ one-to-one corresponds to $\{ x_i : x_i^2=0 ~\text{in}~ H^*(X) \}$.

Now, we consider the quotient ring $H^*(X) / (x_{\alpha})$ where $x_{\alpha}^2=0$ in $H^*(X)$. 
Then we have $n_{i\alpha}=0$ for each $i=1,\ldots,\alpha-1$. 
By $H^*(X) / (x_{\alpha}) \cong \mathbb{Z}[x_1, \ldots, x_d] / (\mathcal{I}_{X}+(x_{\alpha}))$, where $\calI_X=I_X+J_X$ in Proposition~\ref{prop:compute},
we see that $H^*(X) / (x_{\alpha}) \cong \mathbb{Z}[x_1, \ldots, \hat{x}_{\alpha}, \ldots, x_d] / \overline{\mathcal{I}}_{X}$, 
where $\overline{\calI}_X$ is the image of $\calI_X$ by the natural projection $\ZZ[x_1, \ldots, x_d] \rightarrow \ZZ[x_1, \ldots, \hat{x}_{\alpha}, \ldots, x_d]$. 
%Since $\mathbb{Z}[x_1, \ldots, \hat{x}_{\alpha}, \ldots, x_m] / \overline{\mathcal{I}}_{X}$ is isomorphic to the cohomology ring of a Fano Bott manifold $\overline{X}$ 
Then $\mathbb{Z}[x_1, \ldots, \hat{x}_{\alpha}, \ldots, x_d] / \overline{\mathcal{I}}_{X}$ is isomorphic to the cohomology ring of a Fano Bott manifold whose primitive relations are 
\begin{align*}
{\bf v}_1^+ + {\bf v}_1^- &= {\bf v}_{\overline{\varphi}(1)}^{\overline{\sigma}(1)}, \\
&\vdots \\
{\bf v}_{\alpha-1}^+ + {\bf v}_{\alpha-1}^- &= {\bf v}_{\overline{\varphi}(\alpha-1)}^{\overline{\sigma}(\alpha-1)}, \\
{\bf v}_{\alpha+1}^+ + {\bf v}_{\alpha+1}^- &= {\bf v}_{\overline{\varphi}(\alpha+1)}^{\overline{\sigma}(\alpha+1)}, \\
&\vdots \\
{\bf v}_d^+ + {\bf v}_d^- &= {\bf v}_{\overline{\varphi}(d)}^{\overline{\sigma}(d)},
\end{align*}
where $\overline{\varphi} : [d] \setminus \{\alpha\} \rightarrow [d+1] \setminus \{1,\alpha\}$ (resp. $\overline{\sigma} : [d] \setminus \{\alpha\} \rightarrow \{\pm\}$) 
is defined by $\overline{\varphi}(i)=\varphi(i)$ (resp. $\overline{\sigma}(i)=\sigma(i)$). 
This implies that $T_{\overline{X}}$ is isomorphic to $T_X \setminus v_{\alpha}$ as signed rooted forests.

Therefore, when $x_{\alpha}^2=0$ or $(x_\alpha-2x_\beta)^2=0$ in $H^*(X)$, 
taking the quotient $H^*(X) / (x_{\alpha})$ (which is isomorphic to $H^*(X)/(x_\alpha-2x_\beta)$) 
is equivalent to a leaf-cutting by $v_\alpha$ (see Remark~\ref{rem:sve}). 
Moreover, Lemma~\ref{sve} says that for two Fano Bott manifolds $X$ and $X'$, 
$H^*(X) \cong H^*(X')$ implies $\{ H^*(X) / (x_i) : x_i^2=0 \} \cong \{ H^*(X') / (x_i') : x_i'^2=0 \}$. 
These mean that the number of leaves of $T_X$ is equal to the maximal basis number of $H^*(X)$. 

The following proposition will play a crucial role in the proof of Theorem~\ref{thm:FanoBott}. 
\begin{prop}\label{prop:tree_cohomology}
Let $X$ and $X'$ be two Fano Bott manifolds. Assume that $H^*(X) \cong H^*(X')$ and let $F:H^*(X) \rightarrow H^*(X')$ be a ring isomorphism. 
\begin{itemize}
\item[(1)] $F$ induces an isomorphism between $T_X$ and $T_{X'}$ as rooted forests. 
\item[(2)] Let $f:T_X \rightarrow T_{X'}$ be an isomorphism induced by $F$. 
Take a rooted subtree $T_0 \subset T_X$ whose root is the same as that of $T_X$. Then $H^*(X_{T_0}) \cong H^*(X_{f(T_0)})$. 
\end{itemize}
\end{prop}
\begin{proof}
(1) Since $\{x_\alpha \in H^*(X) : x_\alpha^2=0\}$ one-to-one corresponds to the leaves of $T_X$, 
and since we may assume that $F$ sends $\{x_\alpha \in H^*(X) : x_\alpha^2=0\}$ to $\{x_\alpha' \in H^*(X') : x_\alpha'^2=0\}$ (see the above discussion), %although it is not necessary that $F$ sends $x_{l_i}$ to $x_{l_i'}'$  
we obtain a bijection $f$ between the leaves of $T_X$ and those of $T_{X'}$ induced by $F$, as discussed above.

Let $v_{l_1},\ldots,v_{l_k}$ be the leaves of $T_X$ and let $v_{l_i'}'=f(v_{l_i})$ ($i=1,\ldots,k$). 
Remark that taking a quotient by $x_{l_i}$ is equivalent to taking a quotient by $x_{l_i'}'$. We prove the assertion by induction on $k$. 

Let $k=1$. Then $F$ induces the correspondence $v_{l_1}$ and $v_{l_1'}'$. 
Consider $\overline{F} : H^*(X)/(x_{l_1}) \rightarrow H^*(X')/(x_{l_1'}')$. Since $\overline{F}$ is also an isomorphism 
and $H^*(X)/(x_{l_1})$ (resp. $H^*(X')/(x_{l_1'}')$) corresponds to a rooted tree $T_X \setminus v_{l_1}$ (resp. $T_{X'} \setminus v_{l_1'}'$) which contains only one leaf, 
we also obtain the correspondence between their leaves, which are the parents of $v_{l_1}$ and $v_{l_1'}'$, respectively. 
By repeating this procedure, we obtain a bijection of all vertices of $T_X$ and $T_{X'}$ which induces an isomorphism between $T_X$ and $T_{X'}$. 

Let $k>1$. Then $F$ gives a bijection between the leaves of $T_X$ and $T_{X'}$. Take one leaf $v_{l_1}$ %and let $v_{l_1'}'$ be the corresponding leaf. 
and consider $\overline{F} : H^*(X)/(x_{l_1}) \rightarrow H^*(X')/(x_{l_1'}')$. Note that $\overline{F}$ gives a bijection between the leaves of $T_X \setminus v_{l_1}$ and $T_{X'} \setminus v_{l_1'}'$. 
Since $\overline{F}$ is induced from $F$, we see that $\overline{F}$ preserves the bijectivity between $\{v_{l_2},\ldots,v_{l_k}\}$ and $\{v_{l_2'}',\ldots,v_{l_k'}'\}$. 
Note that the number of leaves of $T_X \setminus v_{l_1}$ is either the same as that of $T_X$ or minus one. 
\begin{itemize}
\item When the number of leaves of $T_{X} \setminus v_{l_1}$ decreases from that of $T_{X}$, by the hypothesis of induction, 
we obtain an isomorphism $g:T_{X} \setminus v_{l_1} \rightarrow T_{X'} \setminus v_{l_1'}'$. 
Let $v$ (resp. $v'$) be the parent of $v_{l_1}$ (resp. $v_{l_1'}'$). Once we can see that 
\begin{align}\label{havetoshow}
g(v)=v', 
\end{align}
by combining the correspondence between $v_{l_1}$ and $v_{l_1'}'$, we obtain an isomorphism $T_{X}$ and $T_{X'}$ induced by $F$. 
\item When the number of leaves of $T_{X} \setminus v_{l_1}$ stays the same, 
we see that new leaves $v_{l''}$ and $v_{l'''}$ appear both in $T_{X} \setminus v_{l_1}$ and $T_{X'} \setminus v_{l_1'}'$, respectively. 
Since $\overline{F}$ still gives a bijection between $\{v_{l_2},\ldots,v_{l_k}\}$ and $\{v_{l_2'}',\ldots,v_{l_k'}'\}$, $v_{l''}$ should correspond to $v_{l'''}$. 
Consider $H^*(X)/(x_{l_1},x_{l''})$ and $H^*(X')/(x_{l_1'}',x_{l'''}')$. By the same discussion, we can repeat this procedure until the number of leaves decreases. Hence, the assertion follows. 
\end{itemize}
Our remaining task is to show \eqref{havetoshow}. Note that there is another leaf, say $v_{l_2}$, which is a descendant of $v$. 
When $v_{l_2}$ is a child of $v$, by the hypothesis of induction, we obtain an isomorphism $h:T_X \setminus v_{l_2} \rightarrow T_{X'} \setminus v_{l_2'}'$. 
%Since $h$ still gives a bijection between $\{v_{l_1},v_{l_3},\ldots,v_{l_k}\}$ and $\{v_{l_1'}',v_{l_3'}',\ldots,v_{l_k'}'\}$, we obtain that $h(v)=v'$. 
Remark that $h$ gives a bijection between $\{v_{l_1},v_{l_3},\ldots,v_{l_k}\}$ and $\{v_{l_1'}',v_{l_3'}',\ldots,v_{l_k'}'\}$ since $h$ comes from the isomorphism between $H^*(X)/(x_{l_2})$ and $H^*(X')/(x_{l'_2}')$. 
Thus, in particular, $h$ sends $v_{l_1}$ to $v_{l_1'}'$. Since $v$ (resp. $v'$) is the unique vertex of $T_X$ (resp. $T_{X'}$) adjacent to $v_{l_1}$ (resp. $v_{l_1'}'$), we obtain that $h(v)=v'$. 
Even if $v_{l_2}$ is not a child of $v$, by removing the leaves until the number of leaves decreases, we obtain the same conclusion. 
Hence, we see that $h(v)=g(v)=v'$. 

\noindent
(2) Let $f:T_X \rightarrow T_{X'}$ be an isomorphism constructed in the above way. 
For any rooted subtree $T_0 \subset T_X$ whose root is the same as that of $T_X$, 
since $$H^*(X_{T_0}) \cong H^*(X)/(x_\alpha : v_\alpha \in V(T_X \setminus T_0))$$ and 
$f$ sends $V(T_X \setminus T_0)$ to $V(T_{X'} \setminus f(T_0))$ by construction of $f$, 
we obtain that $H^*(X_{T_0}) \cong H^*(X_{f(T_0)})$, as required. 
\end{proof}

\bigskip

\section{Three operations on matrices and Fano Bott equivalence}\label{sec:matrix}

In this section, we introduce the equivalence relation in $\calF\calB(d)$ (see Definition~\ref{def:FB(d)}), which we call Fano Bott equivalence. 
The notion of Fano Bott equivalence is derived from Bott equivalence defined in \cite{CMO17}.

\begin{defi}[Three operations on $\calF\calB(d)$]
For $A \in \ZZ^{d \times d}$ and $i \in [d]$, let $A^i$ (resp. $A_i$) denote the $i$-th row (resp. column) vector of $A$. 
Let ${\bf e}_i$ (resp. ${\bf e}^i$) denotes the $i$-th unit column (resp. row) vector. 
Given a permutation $\pi$ on $[d]$, we define a permutation matrix $P$ by setting $P_j=\eb_{\pi(j)}$ for each $j$. 
%$$P_j^i=\begin{cases}
%1 &\text{ if }i=\pi(j), \\
%0 &\text{ otherwise}. 
%\end{cases}$$

\bigskip

\noindent
{\bf (Op$1$)}: For a permutation matrix $P$ corresponding to a permutation $\pi$ on $[d]$, we define a map $\Phi_P : \ZZ^{d \times d} \rightarrow \ZZ^{d \times d}$ by $\Phi_P(A):=PAP^{-1}$ for $A \in \ZZ^{d \times d}$. 
Namely, we see that $A_j^i=(\Phi_P(A))_{\pi(j)}^{\pi(i)}$. 

\bigskip

\noindent
{\bf (Op$2^\pm$)}: For $k \in [d]$, we define a map $\Phi_k^\pm:\ZZ^{d \times d} \rightarrow \ZZ^{d \times d}$ as follows: 
for $A \in \ZZ^{d \times d}$, we multiply $-1$ to the $k$-th column and add the $k$-th column times $(k,j)$-entry to the $j$-th column for each $j \in [d] \setminus \{k\}$. 
Namely, we see that $$(\Phi_k^\pm(A))_j=\begin{cases}
-A_k, &\text{ if }j=k, \\
A_j+A^k_j A_k, &\text{ otherwise}. 
\end{cases}$$ 
Notice that $\Phi_k^\pm$ is nothing but a kind of unimodular transformation. 

\bigskip

\noindent
{\bf (Op$3^\pm$)}: Given $A \in \ZZ^{d \times d}$, assume that $A^\ell={\bf 0}$ and $A^k=\pm{\bf e}^\ell$ for some $k$ and $\ell$. 
Then we define $\Phi^\pm_{k,\ell}(A)$ as follows: we multiply $-1$ to $A^k_\ell$, and multiply $A_k^i$ to $A_\ell^i$ if $A_k^i \neq 0$ for $i \in [d] \setminus \{\ell\}$, 
and the other entries stay the same. Namely, we see that $$(\Phi_{k,\ell}^\pm(A))_\ell^i=\begin{cases}
-A_\ell^k, &\text{ if }i=k, \\
A_k^i \cdot A_\ell^i, &\text{ if }A_k^i \neq 0, \\
A_\ell^i, &\text{ otherwise}, 
\end{cases}\;\;\text{ and }(\Phi_{k,\ell}^\pm(A))_j=A_j \text{ for }j \neq \ell. $$
\end{defi}
\begin{defi}[Fano Bott equivalence]
Given two matrices $A$ and $A'$ in $\FB(d)$, we say that $A$ and $A'$ are \textit{Fano Bott equivalent} 
if $A$ can be transformed into $A'$ through a sequence of the three operations (Op$1$), (Op$2^\pm$) and (Op$3^\pm$). 
\end{defi}

\begin{prop}
Let $A \in \FB(d)$. Then $\Phi_k^\pm(A) \in \FB(d)$ for any $k \in [d]$, and 
$\Phi_{k,\ell}^\pm(A) \in \FB(d)$ for any $k,\ell \in [d]$ satisfying $A^\ell={\bf 0}$ and $A^k={\bf e}^\ell$. 
\end{prop}
\begin{proof}
Given $A \in \FB(d)$, we prove that $\Phi_k^\pm(A)$ and $\Phi_{k,\ell}^\pm(A)$ satisfy the conditions in Theorem~\ref{suyama}. 

For each $p \in [d]$, it follows from Theorem~\ref{suyama} that we have $A^p={\bf 0}$, or $A^p={\bf e}^q$ for some $q$ with $p < q \leq d$, 
or $A^p=-{\bf e}^q+A^q$ for some $q$ with $p < q \leq d$. 

Let $A^p={\bf 0}$. Then we easily see that $(\Phi_k^\pm(A))^p=(\Phi_{k,\ell}^\pm(A))^p={\bf 0}$. 

Let $A^p={\bf e}^q$ for some $q$ with $p < q \leq d$. 
\begin{itemize}
\item We see that $(\Phi_k^\pm(A))^p={\bf e}^q$ if $k \neq q$ and $(\Phi_k^\pm(A))^p=-{\bf e}^q+A^q$ if $k=q$. 
\item Note that $(\Phi_{k,\ell}^\pm(A))^\ell={\bf 0}$. We see that $(\Phi_{k,\ell}^\pm(A))^p={\bf e}^q$ if $q \neq \ell$. 
When $q=\ell$, if $p \neq k$ then $A^p_k = 0$. If $p=k$, then $(\Phi_{k,\ell}^\pm(A))^p=-\eb^q$, so the assertion holds. 
\end{itemize}

Let $A^p=-{\bf e}^q+A^q$ for some $q$ with $p< q \leq d$. Then $A^p_j=A^q_j$ holds for any $j \neq q$. 
\begin{itemize}
\item We see that $(\Phi_k^\pm(A))^p=-{\bf e}^q+(\Phi_k^\pm(A))^q$ if $k \neq q$ and $(\Phi_k^\pm(A))^p=\eb^q$ if $k=q$. 
%$$(\Phi_k^\pm(A))_j^p=\begin{cases}
%1, \;&\text{ when }j=k
%A^p_j+A^q_jA^p_q=A^p_j-A^q_j=0, &\text{ when }j \neq k
%\end{cases}$$
\item If $q=\ell$, then $A^q=A^\ell={\bf 0}$. Thus, $(\Phi_{k,\ell}^\pm(A))^p=-\eb^q$ (resp. $=\eb^q$) when $p \neq k$ (resp. $p = k$). 
If $q \neq \ell$, then we see that $(\Phi_{k,\ell}^\pm(A))^p=-\eb^q+(\Phi_{k,\ell}^\pm(A))^q$. 
\end{itemize}
\end{proof}
Remark that (Op$1$) does not necessarily preserve $\FB(d)$, while this corresponds to an isomorphism of the associated signed rooted forests.

\begin{ex}
Let us consdier $A=\begin{pmatrix}
0 &0 &1 &0 &0 &0 \\
0 &0 &-1 &0 &0 &-1 \\
0 &0 &0 &0 &0 &-1 \\
0 &0 &0 &0 &-1 &1 \\
0 &0 &0 &0 &0 &1 \\
0 &0 &0 &0 &0 &0 
\end{pmatrix}$. 
Then it is straightforward to check that $A \in \calF\calB(6)$. Moreover, we see the following: 
\begin{align*}
&\Phi_3^\pm(A)=\begin{pmatrix}
0 &0 &-1 &0 &0 &-1 \\
0 &0 &1 &0 &0 &0 \\
0 &0 &0 &0 &0 &-1 \\
0 &0 &0 &0 &-1 &1 \\
0 &0 &0 &0 &0 &1 \\
0 &0 &0 &0 &0 &0 
\end{pmatrix}, \; 
\Phi_5^\pm(A)=\begin{pmatrix}
0 &0 &1 &0 &0 &0 \\
0 &0 &-1 &0 &0 &-1 \\
0 &0 &0 &0 &0 &-1 \\
0 &0 &0 &0 &1 &0 \\
0 &0 &0 &0 &0 &1 \\
0 &0 &0 &0 &0 &0 
\end{pmatrix}, \\
&\Phi^\pm_{3,6}(A)=\begin{pmatrix}
0 &0 &1 &0 &0 &0 \\
0 &0 &-1 &0 &0 &1 \\
0 &0 &0 &0 &0 &1 \\
0 &0 &0 &0 &-1 &1 \\
0 &0 &0 &0 &0 &1 \\
0 &0 &0 &0 &0 &0 
\end{pmatrix}, \;
\Phi^\pm_{5,6}(A)=\begin{pmatrix}
0 &0 &1 &0 &0 &0 \\
0 &0 &-1 &0 &0 &-1 \\
0 &0 &0 &0 &0 &-1 \\
0 &0 &0 &0 &-1 &-1 \\
0 &0 &0 &0 &0 &-1 \\
0 &0 &0 &0 &0 &0 
\end{pmatrix}. 
\end{align*}
\end{ex}

Let $X$ be $d$-stage Fano Bott manifolds. We discuss the relationship between $A(X)$ and $T_X$. 
As mentioned above, (Op$1$) on $A(X)$ corresponds to an isomorphism of $T_X$. 
More precisely, given a permutation matrix $P$ associated to a permutation $\pi$ on $[d]$, 
$\Phi_P$ corresponds to a bijection on $V(T_X)$ associated to $\pi$. 
For clarifying the relationship between (Op$2^\pm$) (and (Op$3^\pm$)) and the operations on $T_X$, 
we introduce a new notion and operation on the signed rooted forest $T_X$. 
\begin{defi}[Equivalent signs]\label{def:equiv_sign}
Let $T$ and $T'$ be signed rooted forests. 
We say that $T$ and $T'$ have \textit{equivalent signs} if $T$ and $T'$ are isomorphic as rooted forests 
and the signs of the edges of $T$ can be transformed, up to isomorphism as rooted forests, into the signs of the edges of $T'$ 
by some replacements of all the signs of the edges $\{v_i,v_j\}$ for all $j \in \varphi^{-1}(i)$ at the same time. 
When this is the case, we write $T \sim T'$. 
\end{defi}
It directly follows from the construction of the signed rooted forests from primitive relations that 
if $T \sim T'$ then the corresponding primitive relations are equivalent. 
\begin{ex}\label{ex:equiv_signs}
Let us consider the following three signed rooted forests $T_1,T_2,T_3$ (but actually, trees). 
\begin{table}[htb]
\begin{tabular}{ccc}
$T_1$ & $T_2$ & $T_3$ \\
\begin{minipage}{0.33\hsize}
\begin{center}
\begin{tikzpicture}[]
\node[level 1, label=above:{$v_5$}]{}
child {
	node[level 2, label=below:{$v_1$}] {}
	edge from parent node[left]{$+$}
}
child {
	node[level 2, label=below:{$v_2$}] {}
	edge from parent node[below right]{$-$}
}
child {
	node[level 2, label=right:{$v_4$}] {}
	child{
		node[level 2, label=below:{$v_3$}] {}
		edge from parent node[right]{$+$}
	}
	edge from parent node[right]{$-$}
};
\end{tikzpicture}
\end{center}
\end{minipage}
&
\begin{minipage}{0.33\hsize}
	\begin{center}
		\begin{tikzpicture}[]
		\node[level 1, label=above:{$v_5$}]{}
		child {
			node[level 2, label=below:{$v_1$}] {}
			edge from parent node[left]{$+$}
		}
		child {
			node[level 2, label=below:{$v_2$}] {}
			edge from parent node[below right]{$-$}
		}
		child {
			node[level 2, label=right:{$v_4$}] {}
			child{
				node[level 2, label=below:{$v_3$}] {}
				edge from parent node[right]{$+$}
			}
			edge from parent node[right]{$+$}
		};
		\end{tikzpicture}
	\end{center}
\end{minipage}
&
\begin{minipage}{0.33\hsize}
	\begin{center}
		\begin{tikzpicture}[]
		\node[level 1, label=above:{$v_5$}]{}
		child {
			node[level 2, label=below:{$v_1$}] {}
			edge from parent node[left]{$+$}
		}
		child {
			node[level 2, label=below:{$v_2$}] {}
			edge from parent node[below right]{$+$}
		}
		child {
			node[level 2, label=right:{$v_4$}] {}
			child{
				node[level 2, label=below:{$v_3$}] {}
				edge from parent node[right]{$+$}
			}
			edge from parent node[right]{$-$}
		};
		\end{tikzpicture}
	\end{center}
\end{minipage}
\end{tabular}
\end{table}
We see that $T_1 \sim T_2$. In fact, by exchanging the signs of the edges $\{v_5,v_1\},\{v_5,v_2\},\{v_5,v_4\}$ 
and exchanging the vertices $v_1$ and $v_2$ by a graph isomorphism, we see that $T_1$ can be transformed into $T_2$. 

On the other hand, $T_1$ and $T_3$ are not equivalent. 
In fact, the signs of the edges $\{v_5,v_1\}$ and $\{v_5,v_2\}$ in $T_1$ should be different, while those in $T_3$ should be the same. 
\end{ex}

The following proposition will be a crucial part of the proof of Theorem~\ref{thm:FanoBott}. 
\begin{prop}\label{prop:FanoBottMatrix}
Let $X$ be a $d$-stage Bott manifold. Let $A=(n_{ij}) \in \FB(d)$ be the upper triangular matrix associated to $X$. 
Let $T=T_X$ with $V(T)=\{v_1,\ldots,v_d\}$ and let $\varphi$ be the map with \eqref{condition}. 
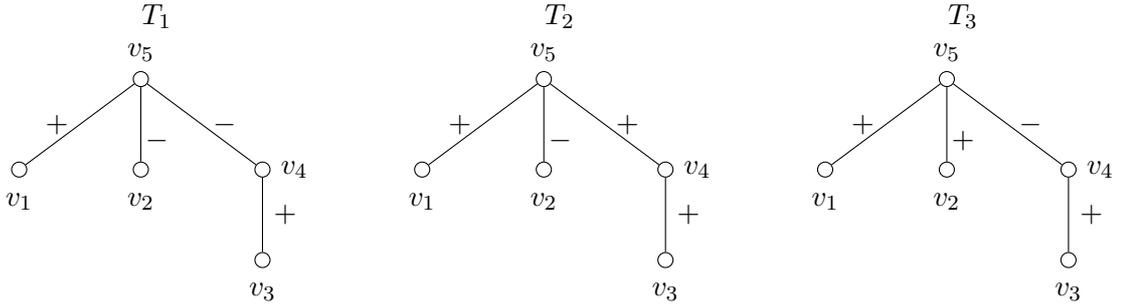
\begin{itemize}
%\item[(1)] Let $A \in \ZZ^{d \times d}$ and assume that $A^i=\eb^j$. 
%Then $(\Phi_k^\pm(A))^i$ and $(\Phi_{k,\ell}^\pm(A))^i$ are either $\eb^j$ or $-\eb^j$ for any $k,\ell$. 
\item[(1)] For $k \in [d]$, the operation $\Phi_k^\pm$ of (Op$2^\pm$) corresponds to 
the simultaneous change of the signs of the edges $\{v_k,v_j\}$ of $T$ for all $j \in \varphi^{-1}(k)$. 
\item[(2)] For $k,\ell \in [d]$, if $A^\ell={\bf 0}$ and $A^k={\bf e}^\ell$, then $v_\ell$ is a root of $T$ and $v_k$ is its child. 
Moreover, the operation $\Phi_{k,\ell}^\pm$ of (Op$3^\pm$) corresponds to the change of the sign of the edge $\{v_\ell,v_k\}$ of $T$. 
\end{itemize}
\end{prop}
\begin{proof}
(1) By \eqref{eq:signed_edge}, it suffices to show that for each $i=1,\ldots,d-1$, 
the left-most nonzero entry of $(\Phi_k^\pm(A))^i$ stays the same as that of $A^i$ when it is not in the $k$-th column, 
while the left-most nonzero entry of $(\Phi_k^\pm(A))^i$ changes from that of $A^i$ when it is in the $k$-th column. 
This directly follows from the definition of the operation $\Phi_k^\pm$. 

\noindent
(2) By our assumption, we know that the left-most nonzero entry of $A^k$ is in the $\ell$-th column. 
Similarly to (1), it suffices to show that for each $i$, the 
the left-most nonzero entry of $(\Phi_{k,\ell}^\pm(A))^i$ stays the same as that of $A^i$ when $i \neq k$, 
while the left-most nonzero entry of $(\Phi_{k,\ell}^\pm(A))^k$ changes from that of $A^k$. 
This also directly follows from the definition of the operation $\Phi_{k,\ell}^\pm$. 
\end{proof}
\begin{rem}
From Proposition~\ref{prop:FanoBottMatrix} (1) together with Proposition~\ref{prop:iso} and Remark~\ref{rem:tree_Bott}, 
we see that given two Fano Bott manifolds, the following three conditions are equivalent: 
\begin{itemize}
\item $X$ and $X'$ are isomorphic as varieties; 
\item $\PC(\Sigma_X)$ and $\PC(\Sigma_{X'})$ are equivalent; 
\item $A(X)$ can be transformed into $A(X')$ through a sequence of the two operations (Op$1$) and (Op$2^\pm$). 
\end{itemize}
\end{rem}

\bigskip

\section{Proof of Theorem~\ref{thm:FanoBott}}\label{sec:proof}
The goal of this section is to complete our proof of Theorem~\ref{thm:FanoBott}. 
Proposition~\ref{leafcut} is an essential part of the proof, 
which is the statement on the distinguishments of cohomology rings by the precise observations of the signed rooted forests associated to Fano Bott manifolds. 

Here, we introduce a special kind of rooted trees: 
\begin{defi}
Let 
        \begin{equation}\label{def:broom}
	\begin{split}
	& V = \{ v_0, v_1, \ldots, v_p, w_1, w_2, \ldots, w_q \} ~~(p \geq 1, \; q \geq 2), \\
	& E = \{\{v_{i-1}, v_i \} : i=1,\ldots,p\} \cup \{\{v_p,w_j\} : j=1,\ldots,q \}.
	\end{split}
	\end{equation}
We call the rooted tree $(V,E,v_0)$ a \textit{broom with $q$ leaves}. 
\end{defi}

\begin{lem}\label{broom}
	Let $B$ and $B'$ be brooms and assign the certain signs of the edges of $B$ and $B'$. 
        Let $X$ and $X'$ be the Fano Bott manifolds corresponding to $B$ and $B'$, respectively. 
        Assume that $H^*(X) \cong H^*(X')$. Then $B \sim B'$. 
\end{lem}
\begin{proof}
Let $B$ and $B'$ be the same broom with $q$ leaves $(V,E,v_0)$ as in \eqref{def:broom}.

First, we prove the assertion in the case $q=2$. 
Then there are only two possibilities $B$ and $B'$ of assignments of signs up to equivalence as follows: 
\begin{align*}
 B&: \; \sign(\{v_{i-1},v_i\})=+ \text{ for }i=1,\ldots,p, \;\; \sign(\{v_p,w_1\})=\sign(\{v_p,w_2\})=+, \\
B'&: \; \sign(\{v_{i-1},v_i\})=+ \text{ for }i=1,\ldots,p, \;\; \sign(\{v_p,w_1\})=+, \; \sign(\{v_p,w_2\})=-. 
\end{align*}
Namely, $B \not\sim B'$. Let $R=H^*(X) \otimes_\ZZ (\ZZ/2)$ and $R'=H^*(X') \otimes_\ZZ (\ZZ/2)$. 
For the convenience to analyze $R$ and $R'$, we rename the vertices as follows: 
$$w_1,w_2,v_p,v_{p-1},\ldots,v_0 \longrightarrow v_1,v_2,v_3,\ldots,v_{p+3}.$$
Once we can prove that $R \not\cong R'$, we conclude $H^*(X) \not\cong H^*(X')$ as required. 

Let $(n_{ij})$ (resp. $(n_{ij}')$) be the upper triangular matrix corresponding to $B$ (resp. $B'$). 
By Remark~\ref{rem:forest_matrix}, we can compute them as follows: 
\begin{align*}
&n_{13}=1, \; n_{i,i+1}=1 \text{ for }i=2,\ldots,p+2, \;\; n_{ij}=0 \text{ otherwise}, \\ 
&n_{23}'=-1, \; n_{13}'=n_{24}'=1, n_{i,i+1}'=1 \text{ for }i=3,\ldots,p+2 \;\; n_{ij}'=0 \text{ otherwise}. 
\end{align*}
Here, we consider the entries of $(n_{ij})$ and $(n_{ij}')$ as the elements of $\ZZ/2$. More concretely, we identify $-1$ with $1$. 
Consider the cut-rank discussed in \cite[Section 8.5]{CMO17}. 
Let $D$ (resp. $D'$) denote the acyclic digraph corresponding to $(n_{ij}) \in (\ZZ/2)^{(p+3) \times (p+3)}$ (resp. $(n_{ij}') \in (\ZZ/2)^{(p+3) \times (p+3)}$). 
In this case, we see that both $L_0(D)$ and $L_0(D')$ correspond to $\{1,2\}$ and 
\begin{align*}
\rho_D(L_0(D))&=\rank\left(\begin{pmatrix}1 &0 &\cdots &0 \\ 1 &0 &\cdots &0 \end{pmatrix}\right)=1, \;\;\text{ and }\\
\rho_{D'}(L_0(D'))&=\rank\left(\begin{pmatrix}1 &0 &\cdots &0 \\ 1 &1 &\cdots &0 \end{pmatrix}\right)=2. 
\end{align*}
Therefore, we obtain that $R \not\cong R'$ by \cite[Theorem 1.1 and Proposition 8.6 (i)]{CMO17}. %Therefore, $H^*(X) \not\cong H^*(X')$. 

Now, let us consider the general case. We assign the signs of the edges $\{v_p,w_j\}$ for $j=1,\ldots,q$ as follows: 
%	\begin{table}[!h]
%		\centering
%		\begin{tabular}{|c|c|c|c|c|c|} \hline
%     signs of $\{v_p,w_j\}$ & ~~$B$~~ & ~~$B'$~~  \\ \hline
%			$+$ & $a$ & $a'$ \\ \hline
%			$-$ & $b$ & $b'$ \\ \hline
%		\end{tabular}
%	\end{table}
\begin{align*}
B: \;\; &|\{j \in [q] : \sign(\{v_p,w_j\})=+\}|=a, \;\; |\{j \in [q] : \sign(\{v_p,w_j\})=-\}|=b, \\
B': \;\; &|\{j \in [q] : \sign(\{v_p,w_j\})=+\}|=a', \;\; |\{j \in [q] : \sign(\{v_p,w_j\})=-\}|=b'. 
\end{align*}
Here, we have $a+b=a'+b'=q$. When $B \not\sim B'$, we have $\{a,b\} \neq \{a',b'\}$. 
Then, for any isomorphism $f:V \rightarrow V$ between signed rooted trees $B$ and $B'$, we can find a pair of leaves $w_i,w_j$ such that 
$$(\sign(\{v_p,w_i\}),\sign(\{v_p,w_j\})) \neq (\sign(\{v_p,f(w_i)\}),\sign(\{v_p,f(w_j)\})).$$ 
Let us fix such an isomorphism $f$ and such leaves $w_i,w_j$. 
Let $B_0$ be the subtree of $B$ induced by $v_0,v_1,\ldots,v_p$ and $w_i,w_j$. In particular, $B_0$ is a broom with $2$ leaves. 
Then it follows from the above discussion that $H^*(X_{B_0}) \not\cong H^*(X_{f(B_0)})$. 
Therefore, Proposition~\ref{prop:tree_cohomology} (2) implies that $H^*(X) \not\cong H^*(X')$, as desired. 
\end{proof}

The following proposition plays the crucial role for the proof of Theorem~\ref{thm:FanoBott}. 

\begin{prop}\label{leafcut}
Let $X$ and $X'$ be Fano Bott manifolds. Assume that $H^*(X) \cong H^*(X')$. 
Then $T_X$ and $T_{X'}$ have equivalent signs by exchanging the signs assigned to the edges adjacent to the roots if necessary. 
\end{prop}
\begin{proof}%[Proof of Proposition~\ref{leafcut}]
By Proposition~\ref{prop:tree_cohomology} (1), there is an isomorphism $f:T_X \rightarrow T_{X'}$ as rooted forests. 

Suppose, on the contrary, that $T_X \not\sim T_{X'}$. Then Definition~\ref{def:equiv_sign} says that 
there is a vertex $v$ of $T_X$ such that the signs of the edges adjacent to $v$ are not preserved by $f$ even if we exchange all those signs at the same time. 
Since the equivalence of signs is considered up to isomorphism as rooted forests, 
we have to treat a certain part of children of $v$ which can be transferred by an isomorphism as rooted forests. 
Let $w_j^{(i)}$ ($i=1,\ldots,p$, $j=1,\ldots,q_i$) be all the children of $v$ such that 
$w_j^{(i)}$ can be transferred into $w_{j'}^{(i')}$ by $f$ if and only if $i=i'$. 
Let $v'=f(v)$, let ${w'}_i^{(j)}=f(w_j^{(i)})$ for each $i$ and $j$, 
let $a_i=|\{ j \in [q_i] : \sign(\{v,w_j^{(i)}\})=+\}|$ (resp. $a_i'=|\{ j \in [q_i] : \sign(\{v',{w'}_j^{(i)}\})=+\}|$) and let $b_i=q_i-a_i$ (resp. $b_i'=q_i'-a_i'$) for each $i$. 
The assumption $T \not\sim T'$ implies that there is $I \subset [p]$ with 
$$\left\{\sum_{i \in I}a_i,\sum_{i \in I}b_i\right\} \neq \left\{\sum_{i \in I}a_i',\sum_{i \in I}b_i'\right\},$$
otherwise it is a contradiction to the choice of $v$. Let us fix such $I$. 
%either (I) or (II) must happen: 
%\begin{itemize}
%\item[(I)] there is $i \in [p]$ with $\{a_i,b_i\} \neq \{a'_i,b'_i\}$; 
%\item[(II)] there are $i,j \in [p]$ with $\{a_i+a_j,b_i+b_j\} \neq \{a'_i+a_j',b'_i+b_j'\}$. 
%\end{itemize}

Note that the length from the root $v_0$ of $T_X$ to $v$ is at least $1$. 
Let $v_0,v_1,\ldots,v_k=v$ be the vertices between $v_0$ and $v$. 
Consider the broom $B$ defined by $v_0,v_1,\ldots,v_k$ and $w_j^{(i)}$ for $i \in I$ and $j \in [q_i]$. Let $B'=f(B)$. 
Note that $B \not\sim B'$. By Lemma~\ref{broom}, we see that $H^*(X_B) \not\cong H^*(X_{B'})$. 

On the other hand, since $B$ and $B'$ are the rooted subtrees of $T_X$ and $T_{X'}$ whose roots are the same as that of $T_X$ and $T_{X'}$, respectively, 
we see that $H^*(X_B) \cong H^*(X_{B'})$ by Proposition~\ref{prop:tree_cohomology} (2), a contradiction. 

Therefore, we conclude that $T_X \sim T_{X'}$, as required. 
\end{proof}

Now, we are ready to give a proof of Theorem~\ref{thm:FanoBott}. 
\begin{proof}[Proof of Theorem~\ref{thm:FanoBott}]
Note that (2) $\Rightarrow$ (3) is trivial and (3) $\Rightarrow$ (1) directly follows from Propositions~\ref{prop:FanoBottMatrix} and \ref{leafcut}. 

In what follows, we prove (1) $\Rightarrow$ (2). 
Let $X$ and $X'$ be $d$-stage Fano Bott manifolds and let $T=T_X$ and $T'=T_{X'}$ be the associated signed rooted forests, respectively. 
By Proposition~\ref{prop:FanoBottMatrix}, the assumption of (1) says that 
$T \cong T'$ as (non-signed) rooted forests and $T$ and $T'$ have equivalent signs by exchanging the signs assigned to the edges adjacent to the roots. 
Namely, there is an isomorphism $f:T \rightarrow T'$ as rooted forests which preserves the signs of all edges of $T$ and $T'$ except for the edges adjacent to roots. 
We assume that $f$ is an ``identity'', more precisely, $f(v_i)=v_i'$, where $V(T)=\{v_1,\ldots,v_d\}$ and $V(T')=\{v_1',\ldots,v_d'\}$. 
        
Let $\Sigma_1=(\Sigma_X)_1=\{{\bf v}_i^\pm : i=1,\ldots,d\}$ and let $\Sigma_1'=(\Sigma_{X'})_1=\{{\bf w}_i^\pm : i=1,\ldots,d\}$. 
Let $A:=A(X)=(n_{ij})_{1 \leq i,j \leq d}$ (resp. $A':=A(X')=(n_{ij}')_{1 \leq i,j \leq d}$). 
As mentioned in Remark~\ref{rem:forest_matrix}, we can read off $(n_{ij})$ (resp. $(n_{ij}')$) from $T$ (resp. $T'$). 
Since ${\bf v}_1^+,\ldots,{\bf v}_d^+$ (resp. ${\bf w}_1^+,\ldots,{\bf w}_d^+$) are linearly independent, we can take ${\bf v}_i^+={\bf w}_i^+={\bf e}^i$ for each $i$. 
Thus, we see that the matrix $M$ (resp. $M'$) whose row vectors consist of ${\bf v}_1^+,\ldots,{\bf v}_d^+,{\bf v}_1^-,\ldots,{\bf v}_d^-$ 
(resp. ${\bf w}_1^+,\ldots,{\bf w}_d^+,{\bf w}_1^-,\ldots,{\bf w}_d^-$) looks as follows: 
\begin{align*}
        M=\begin{pmatrix}
         & & & & \\
	 & &E& & \\
	 & & & & \\
	-1 & n_{12} & n_{13}  & \cdots & n_{1d} \\
	   & -1        & n_{23} & \cdots & n_{2d} \\
	   &            & -1        & \cdots & n_{3d} \\
	   &            &            & \ddots & \vdots \\
           &            &            &            & -1        
        \end{pmatrix}\text{ and }
        M'=\begin{pmatrix}
         & & & & \\
	 & &E& & \\
	 & & & & \\
	-1 & n_{12}' & n_{13}'  & \cdots & n_{1d}' \\
	   & -1        & n_{23}' & \cdots & n_{2d}' \\
	   &            & -1        & \cdots & n_{3d}' \\
	   &            &            & \ddots & \vdots \\
           &            &            &            & -1        
        \end{pmatrix}, 
\end{align*}
where $E \in \{0,1\}^{d \times d}$ denotes the identity matrix. 
Namely, the top halves of $M$ and $M'$ are $E$ and the bottom half of $M$ (resp. $M'$) is $-E+A$ (resp. $-E+A'$).

Here, we see that the operation (Op$1$) is a unimodular transformation on $M$. In fact, (Op$1$) permutes the corresponding columns and the corresponding rows of the top half and the bottom half. 
Moreover, since (Op$2^\pm$) is a unimodular transformation expressed by $T_k:=\begin{pmatrix} \eb^1 \\ \vdots \\ -\eb^k+A^k \\ \vdots \\ \eb^d\end{pmatrix} \in \ZZ^{d \times d}$, 
we see that $E \cdot T_k = \begin{pmatrix} \eb^1 \\ \vdots \\ -\eb^k+A^k \\ \vdots \\ \eb^d\end{pmatrix}$ and \begin{align*}
(-E+A) \cdot T_k = -T_k + \Phi_k^\pm(A) = \begin{pmatrix} -\eb^1+\Phi_k^\pm(A)^1 \\ \vdots \\ \eb^k \\ \vdots \\ -\eb^d+\Phi_k^\pm(A)^d\end{pmatrix}. 
\end{align*}
Let $M''$ be the resulting matrix after applying a certain sequence of unimodular transformations corresponding to (Op$1$) and (Op$2^\pm$) to $M$ and let $T''$ be the signed rooted forest associated to (the bottom half of) $M''$. 
Since $A$ and $A'$ are Fano Bott equivalent, we may assume that $M''$ can be transform into $M'$ by a certain sequence of transformations corresponding to (Op$3^\pm$).

First, we assume that $T$ and $T'$ are trees, i.e., connected. Let $v_d$ be the unique root of $T$. 
By our assumption, we see that $M''$ and $M'$ agree except for the $d$-th columns. (See Proposition~\ref{prop:FanoBottMatrix}.) 
Let $v_{i_1},\ldots,v_{i_k}$ be the children of the root $v_d$ such that 
the signs of $\{v_d,v_{i_j}\}$ and $\{v_d',v_{i_j}'\}$ are different. 
Then, let us consider the subtrees $T_{i_j}$ of $T$ induced by all descendants of $v_{i_j}$ with its root $v_{i_j}$. 
By the construction of $(n_{ij})$ from $T$, we see that $n_{pq} \neq 0$ only if both $v_p$ and $v_q$ are contained in the same $T_{i_j}$ for some $j$. 
Now, apply the unimodular transformations $I_j=(e_{pq}^{(j)})_{1 \leq p,q \leq d} \in \ZZ^{d \times d}$ for $j=1,\ldots,k$ to $M''$, where 
$$e_{pp}^{(j)}=\begin{cases}
-1 &\text{ if } v_p \in V(T_{i_j}), \\
1 &\text{ otherwise}, 
\end{cases}$$
and $e_{pq}^{(j)}=0$ if $p \neq q$. Then we see that the row vectors of $M'' \cdot (I_1 \cdots I_k)$ and $M'$ coincide up to sign. 
Hence, those satisfy the condition in Lemma~\ref{key}, so we conclude that $X$ and $X'$ are diffeomorphic. 

Even if $T$ and $T'$ have at least two connected components, we may apply the above procedure to each connected component. 
\end{proof}
\begin{rem}\label{rem:(4)}
By the above proof, we see that the three conditions of Theorem~\ref{thm:FanoBott} is equivalent to the fourth condition: 
\begin{itemize}
\item[(4)] the signed rooted forests $T_X$ and $T_{X'}$ have equivalent signs by exchanging the signs assigned to the edges adjacent to the roots if necessary.
\end{itemize}
\end{rem}

\begin{ex}
Let us consider the following signed rooted forests $T$ and $T'$. 
\begin{table}[!h]
\begin{tabular}{ccc}
$T$ & $T'$ & $T''$ \\
\begin{minipage}{0.33\hsize}
\begin{center}
\begin{tikzpicture}[]
\node[level 1, label=above:{$v_7$}]{}
child {
	node[level 2, label=left:{$v_3$}] {}
	child{
		node[level 3, label=below:{$v_1$}] {}
		edge from parent node[left]{$+$}
	}
	child{
		node[level 3, label=below:{$v_2$}] {}
		edge from parent node[right]{$-$}
	}
	edge from parent node[left]{$+$}
}
child {
	node[level 2, label=right:{$v_6$}] {}
	child{
		node[level 3, label=below:{$v_4$}] {}
		edge from parent node[left]{$+$}
	}
	child{
		node[level 3, label=below:{$v_5$}] {}
		edge from parent node[right]{$+$}
	}
	edge from parent node[right]{$+$}
};
\end{tikzpicture}
\end{center}
\end{minipage}
&
\begin{minipage}{0.33\hsize}
\begin{center}
\begin{tikzpicture}[]
\node[level 1, label=above:{$v_7$}]{}
child {
	node[level 2, label=left:{$v_3$}] {}
	child{
		node[level 3, label=below:{$v_1$}] {}
		edge from parent node[left]{$-$}
	}
	child{
		node[level 3, label=below:{$v_2$}] {}
		edge from parent node[right]{$+$}
	}
	edge from parent node[left]{$+$}
}
child {
	node[level 2, label=right:{$v_6$}] {}
	child{
		node[level 3, label=below:{$v_4$}] {}
		edge from parent node[left]{$-$}
	}
	child{
		node[level 3, label=below:{$v_5$}] {}
		edge from parent node[right]{$-$}
	}
	edge from parent node[right]{$-$}
};
\end{tikzpicture}
\end{center}
\end{minipage}
&
\begin{minipage}{0.33\hsize}
\begin{center}
\begin{tikzpicture}[]
\node[level 1, label=above:{$v_7$}]{}
child {
	node[level 2, label=left:{$v_3$}] {}
	child{
		node[level 3, label=below:{$v_1$}] {}
		edge from parent node[left]{$-$}
	}
	child{
		node[level 3, label=below:{$v_2$}] {}
		edge from parent node[right]{$+$}
	}
	edge from parent node[left]{$+$}
}
child {
	node[level 2, label=right:{$v_6$}] {}
	child{
		node[level 3, label=below:{$v_4$}] {}
		edge from parent node[left]{$-$}
	}
	child{
		node[level 3, label=below:{$v_5$}] {}
		edge from parent node[right]{$-$}
	}
	edge from parent node[right]{$+$}
};
\end{tikzpicture}
\end{center}
\end{minipage}
\end{tabular}
\end{table}
${}$\\
We see that $T \not\sim T'$ by the difference of signs of edges $\{v_7,v_3\}$ and $\{v_7,v_6\}$, 
but the corresponding Fano Bott manifolds are diffeomorphic since the remaining signs can be transformed. 
The following matrices are $M,M''$ and $M'$ appearing in the above proof and 
the unimodular transformations $I_1$ which corresponds to the subtree induced by the descendants of $v_6$ used in the proof: 
\begin{align*}
&M=\begin{pmatrix}
  & & &E& & & \\
  & & & & & & \\
-1&  &1 & & & & \\
  &-1&-1& & & &1 \\
  &  &-1& & & &1\\
  &  & &-1& &1& \\
  &  & & &-1&1& \\
  &  & & & &-1&1\\
  &  & & & & &-1\\
\end{pmatrix}, \;\;
M''=\begin{pmatrix}
  & & &E& & & \\
  & & & & & & \\
-1& &-1& & & &1\\
  &-1&1& & & &\\
  & &-1& & & &1\\
  & & &-1& &-1&1 \\
  & & & &-1&-1&1 \\
  & & & & &-1&1\\
  & & & & & &-1\\
\end{pmatrix}, \\
&I_1=
\begin{pmatrix}
1 &  &  & & & & \\
   &1&  & & & & \\
   &  &1 & & & & \\
   &  &  &-1& & & \\
   &  &  & &-1& & \\
   &  &  & & &-1& \\
   &  &  & & & &1 
\end{pmatrix}, \;\;
M'=\begin{pmatrix}
  & & &E& & & \\
  & & & & & & \\
-1& &-1& & & &1\\
  &-1&1& & & & \\
  & &-1& & & &1\\
  & & &-1& &-1&-1\\
  & & & &-1&-1&-1\\
  & & & & &-1&-1\\
  & & & & & &-1\\
\end{pmatrix}. 
\end{align*}
We can transform $M$ into $M''$ by applying the permutation matrix corresponding to the transposition $(1,2)$ and $T_6$ appearing in the above proof. 
Note that $T''$ above is the signed rooted tree corresponding to $M''$. 

We can see that $M''I_1$ coincides with $M'$ up to signs of rows. 
In fact, the first, second, third and seventh rows of the top and bottom halves of $M''I_1$ are exactly equal to those of $M'$, 
while the fourth, fifth and sixth rows are equal up to signs. 
\end{ex}

\end{document}